\documentclass[11pt]{amsart}
\usepackage{geometry}                
\usepackage{verbatim}
\usepackage{graphicx}
\usepackage{amssymb}
\usepackage{epstopdf}
\newtheorem{theorem}{Theorem}[section]
\newtheorem{lemma}[theorem]{Lemma}

\newtheorem{prop}[theorem]{Proposition}
\theoremstyle{definition}

\def\C{\Gamma}

\def\eproof{$\Box$ \medskip}
\def\cal{\mathcal}
\def\Tr{{\rm Tr}}

\newcommand{\JJ}{{\bf J}}
\newcommand{\CC}{\mathbb C}
\newcommand{\Sph}{\mathbb S}
\newcommand{\R}{\mathbb R}

\newcommand{\N}{\mathbb N}
\newcommand{\Z}{\mathbb Z}

\newcommand{\Hp}{{\mathbb H}^2}

\title{Simple length rigidity for Kleinian surface groups and applications}
\author{Martin Bridgeman}
\address{Boston College}
\author{Richard D. Canary}
\address{University of Michigan}
\date{\today}
\thanks{Bridgeman was partially suppported by grant DMS-1500545 and
Canary was partially supported by  grant DMS-1306992, from the National Science Foundation.
The authors  also acknowledge  support from U.S. National Science Foundation grants 
DMS 1107452, 1107263, 1107367 ``RNMS: GEometric structures And Representation varieties" (the GEAR Network).}

\begin{document}
\begin{abstract}
We prove that a Kleinian surface groups is determined, up to conjugacy in the isometry group of $\mathbb H^3$,
by its simple marked length spectrum.
As a first application, we show that a discrete faithful representation of
the fundamental group of a compact, acylindrical, hyperblizable 3-manifold $M$ is similarly determined
by the translation lengths of  images of elements of $\pi_1(M)$ represented
by simple curves on the boundary of $M$. As a second application, we show the group of
diffeomorphisms of quasifuchsian space which preserve the renormalized pressure intersection  is
generated by the (extended) mapping class group and complex conjugation.
\end{abstract}
\maketitle

\section{Introduction}

We show that  if $\rho_1$ and $\rho_2$ are two discrete, faithful representations of a surface group $\pi_1(S)$ into 
${\rm PSL}(2,\mathbb C)$ with the same simple marked length spectrum, then $\rho_1$ is either conjugate 
to $\rho_2$ or its complex conjugate.
(Two such representations have the same simple marked length spectrum if whenever $\alpha\in\pi_1(S)$ is
represented by a simple closed curve, then the images of $\alpha$ have the same translation length. The complex
conjugate of a representation is obtained by conjugating the representation by $z\to\bar z$.)
March\'e and Wolff \cite[Sec. 3]{marche-wolff} have exhibited non-elementary
representations of a closed surface group of genus two
into ${\rm PSL}(2,\mathbb R)$ with the same simple marked length spectrum which do not have the same
marked length spectrum, so the corresponding statement does not hold for non-elementary representations.

We give two applications of our main result. First, if $M$ is a compact, acylindrical, hyperbolizable 3-manifold, we show that if
$\rho_1$ and $\rho_2$ are  discrete faithful representations  of $\pi_1(M)$ into ${\rm PSL}(2,\mathbb C)$
such that translation
lengths of the images of elements of $\pi_1(M)$ corresponding to simple curves in the boundary of $M$ agree,
then $\rho_1$ is either conjugate to $\rho_2$ or its complex conjugate.
For our second application we consider the renormalized pressure intersection,
first defined by Burger \cite{burger-renorm} and  further studied by Bridgeman-Taylor \cite{BT08}.
Bridgeman \cite{Bri10} (see also \cite{BCLS}) 
showed that the Hessian of the renormalized pressure intersection  gives rise to a 
path metric on quasifuchsian space $QF(S)$. We show that the group of diffeomorphisms of $QF(S)$
which preserve the renormalized pressure intersection  is generated by the (extended) mapping class group
and the involution of $QF(S)$ determined by complex conjugation.

\subsection{Simple length rigidity for Kleinian surface groups}
A {\em Kleinian surface group} is a discrete, faithful representation
$\rho:\pi_1(S)\to {\rm PSL}(2,\mathbb C)$ where 
$S$ is  a closed, connected, orientable surface of genus at least two. 
If $\alpha\in\pi_1(S)$, then let  $\ell_\rho(\alpha)$ denote the translation of length of $\rho(\alpha)$,
or equivalently the length of the closed geodesic in the homotopy class of $\alpha$ in the quotient
hyperbolic 3-manifold $\mathbb H^3/\rho(\pi_1(S))$.
We say that two Kleinian surface groups $\rho_1:\pi_1(S)\to {\rm PSL}(2,\mathbb C)$   and
$\rho_2:\pi_1(S)\to {\rm PSL}(2,\mathbb C)$ have the 
{\em same marked  length spectrum}
if $\ell_{\rho_1}(\alpha)=\ell_{\rho_2}(\alpha)$ for all $\alpha\in\pi_1(S)$. Similarly, we say that $\rho_1$ and $\rho_2$
have the {\em same simple marked  length spectrum}
if $\ell_{\rho_1}(\alpha)=\ell_{\rho_2}(\alpha)$ whenever $\alpha$ has a representative on $S$ which is a simple closed curve.
If $\rho:G\to {\rm PSL}(2,\mathbb C)$ is a representation we define its {\em complex conjugate} $\bar\rho_2$ to
be the representation obtained by conjugating by $z\to\bar z$.

\begin{theorem}{}
\label{main}
{\rm (Simple length rigidity for Kleinian surface groups)} If $S$ is a closed, connected,
orientable surface of genus at least two, and $\rho_1:\pi_1(S)\to {\rm PSL}(2,\mathbb C)$ and 
$\rho_2:\pi_1(S)\to {\rm PSL}(2,\mathbb C)$ are Kleinian surface groups with the same
simple marked length spectrum, then $\rho_1$ is conjugate to either $\rho_2$ or $\bar\rho_2$.
\end{theorem}

Since the full isometry group of $\mathbb H^3$ may be identified with the group generated by ${\rm PSL}(2,\mathbb C)$,
regarded as the group of fractional linear transformations, and $z\to \bar z$, one may reformulate our
main result as saying that two Kleinian surface groups with the same simple marked length spectrum are
conjugate in the isometry group of $\mathbb H^3$.
 
\medskip\noindent
{\bf Historical remarks:}
It is a classical consequence of the Fenchel-Nielsen coordinates for Teichm\"uller space that there are 
finitely many simple curves on $S$ whose lengths determine a Fuchsian 
(i.e. discrete and faithful) representation of $\pi_1(S)$  into ${\rm PSL}(2,\mathbb R)$ 
up to
conjugacy in ${\rm PGL}(2,\mathbb R)$, which we may identify with the isometry group of $\mathbb H^2$.
However, March\'e and Wolff \cite[Sec. 3]{marche-wolff} showed that there exist
non-Fuchsian representations of the fundamental group of a surface of genus two into ${\rm PSL}(2,\mathbb R)$
with the same simple marked length spectrum which do
not have the same marked length spectrum. The representations constructed by March\'e and Wolff do
not lift to ${\rm SL}(2,\mathbb R)$, so do not lie in the same component of the ${\rm PSL}(2,\mathbb C)$-representation variety
as the discrete faithful representations.

Kourounitis \cite{kourounitis} showed that there are finitely many simple curves on $S$ whose complex
lengths  (see section \ref{complex length} for a discussion of complex length)
determine a quasifuchsian surface group up to conjugacy in ${\rm PSL}(2,\mathbb C)$. 
Culler and Shalen \cite[Prop. 1.4.1]{culler-shalen} showed that there are finitely many curves
whose traces determine a non-elementary representation into ${\rm SL}(2,\mathbb C)$, 
up to conjugacy in ${\rm SL}(2,\mathbb C)$, while 
Charles-March\'e \cite[Thm. 1.1]{charles-marche} showed that one may choose the finite collection to
consist of simple closed curves.

Kim \cite{kim-rigid} previously showed that two Zariski dense representations into the isometry group ${\rm Isom}(X)$
of a rank one symmetric space $X$ with the same full marked length spectrum are conjugate in ${\rm Isom}(X)$.
More generally, Dal'Bo and Kim \cite{dalbo-kim} showed that  any surjective homomorphism between Zariski
dense subgroups of a  semi-simple Lie group $G$, with trivial center and no compact factors, which
preserves translation length on the associated symmetric space extends to an automorphism of $G$.

\subsection{Simple length rigidity for acylindrical hyperbolic 3-manifolds}

A compact, orientable 3-manifold $M$ with non-empty boundary is said to be {\em hyperbolizable} if its interior admits
a complete hyperbolic metric, which implies that there exists
a discrete, faithful representation of $\pi_1(M)$ into ${\rm PSL}(2,\mathbb C)$.
A compact, hyperbolizable 3-manifold 
is said to be {\em acylindrical} if every $\pi_1$-injective proper map of an annulus into $M$ is
properly homotopic into the boundary of $M$. (Recall that a map of a surface into a $3$-manifold
is said to be proper if it maps the boundary of the surface into the boundary of 3-manifold and
that a proper homotopy is a homotopy through proper maps.)

In this setting, we use Theorem \ref{main} show that a discrete, faithful
representation of $\pi_1(M)$ into ${\rm PSL}(2,\mathbb C)$ is determined, up to conjugacy in
the isometry group of $\mathbb H^3$, by the translation lengths of images of simple  curves in the boundary
$\partial M$ of $M$.

\begin{theorem}
\label{acyl app}
If $M$ is a compact, acylindrical, hyperbolizable 3-manifold, and \hbox{$\rho_1:\pi_1(M)\to {\rm PSL}(2,\mathbb C)$} and
$\rho_2:\pi_1(M)\to {\rm PSL}(2,\mathbb C)$ are two discrete faithful representations, such
that $\ell_{\rho_1}(\alpha)=\ell_{\rho_2}(\alpha)$ if $\alpha\in\pi_1(M)$ is represented by a simple
closed curve on $\partial M$, then $\rho_1$ is conjugate to either $\rho_2$ or $\bar\rho_2$.
\end{theorem}

\subsection{Isometries of the renormalized pressure intersection}
Burger \cite{burger-renorm} introduced a renormalized pressure intersection between
convex cocompact representations into rank one Lie groups. Bridgeman and Taylor \cite{BT08}
extensively studied this renormalized pressure intersection  in the setting of quasifuchsian
representation. We say that \hbox{$\rho:\pi_1(S)\to {\rm PSL}(2,\mathbb C)$} is {\em quasifuchsian}
if it is topologically conjugate, in terms of its action on $\widehat{\mathbb C}$,
to a Fuchsian representation into ${\rm PSL}(2,\mathbb R)$. If $T>0$ we let
$$R_T(\rho)=\{[\alpha]\in [\pi_1(S)]\ |\ \ell_\rho(\alpha)\le T\}$$
where $[\pi_1(S)]$ is the set of conjugacy classes in $\pi_1(S)$. We define the {\em entropy}
$$h(\rho)=\limsup\frac{\log(\#(R_T(\rho)))}{T}$$
of a quasifuchsian representation $\rho$. Sullivan \cite{sullivan-entropy} showed that
$h(\rho)$ is the Hausdorff dimension of the limit set of $\rho(\pi_1(S))$. 

Let $QF(S)$ denote the space of ${\rm PSL}(2,\mathbb C)$-conjugacy classes of quasifuchsian representations.
Bers \cite{bers-simultaneous} showed that $QF(S)$ is an analytic manifold which may
be naturally identified with $\mathcal{T}(S)\times\mathcal T(S)$. If $\rho_1,\rho_2\in QF(S)$,
the {\em renormalized pressure intersection} of $\rho_1$ and $\rho_2$ is given by
$$\JJ(\rho_1,\rho_2)=\frac{h(\rho_2)}{h(\rho_1)}\lim_{T\to\infty}\left(\frac{1}{\#(R_T(\rho))}\sum_{[\alpha]\in R_T(\rho_1)}
\frac{\ell_{\rho_2}(\alpha)}{\ell_{\rho_1}(\alpha)}\right).$$
Bridgeman  and Taylor \cite{BT08} showed that the Hessian of $J$ gives rise to a non-negative bilinear form
on  the tangent space${\sf{T}}QF(S)$ of quasifuchsian space, called the pressure form. 
Motivated by work of McMullen \cite{mcmullen-pressure} in the setting of Teichm\"uller space, Bridgeman \cite{Bri10}
used the thermodynamic formalism to show that
the only degenerate vectors for the pressure form correspond to pure
bending at points on the Fuchsian locus. Moreover, the pressure form gives rise to a path metric
on $QF(S)$, called the {\em pressure metric} (see also  \cite[Cor. 1.7]{BCLS}).

We say a smooth immersion $f:QF(S) \rightarrow QF(S)$ is a {\em smooth isometry  of the renormalized pressure intersection}  if
$$\JJ(f(\rho_1),f(\rho_2)) = \JJ(\rho_1,\rho_2)$$ for all $\rho_1,\rho_2\in QF(S)$.
We recall that the (extended) mapping class group ${\rm Mod}^*(S)$ is the group of
isotopy classes of homeomorphisms of $S$.
Since $J$ is invariant under the action of ${\rm Mod}^*(S)$, every
element of ${\rm Mod}^*(S)$ is a smooth isometry of the renormalized
pressure intersection. There exists an involution $\tau:QF(S)\to QF(S)$ given by taking $[\rho]$ to $[\bar\rho]$.
Since $\tau$ preserves the marked length spectrum, it is  an isometry of the renormalized
pressure intersection. We use our main result and work of Bonahon \cite{Bon88} to show
that these give rise to all smooth isometries of the renormalized pressure intersection.

\begin{theorem}
\label{isom}
If  $S$ is a closed, orientable  surface of genus at least two, then the group of
smooth isometries of the renormalized pressure intersection  on $QF(S)$ is generated by the (extended)
mapping class group ${\rm Mod}^*(S)$ and complex conjugation $\tau$.
\end{theorem}

\medskip\noindent
Royden \cite{royden} showed that ${\rm Mod}^*(S)$ is the isometry group of the Teichm\"uller metric on $T(S)$. 
Masur and Wolf \cite{masur-wolf} proved that ${\rm Mod}^*(S)$ is the isometry group of the Weil-Petersson metric on $T(S)$.
Bridgeman \cite{Bri10} used work of Wolpert \cite{Wol86} to show that  
the restriction of the pressure form to the Fuchsian locus is a multiple of the Weil-Petersson metric.

One may thus view Theorem \ref{isom} as evidence in favor of the following natural conjecture.

\medskip\noindent
{\bf Conjecture:} {\em The isometry group of the pressure metric on quasifuchsian space $QF(S)$ is
generated by the (extended) mapping class group and complex conjugation.}

\medskip

In the proof of Theorem \ref{isom}, we establish the following strengthening of our
main result which may be of independent interest.

\begin{theorem}
\label{ksimple rigidity}
If $S$ is a closed, connected, orientable surface of genus at least two,
\hbox{$\rho_1:\pi_1(S)\to {\rm PSL}(2,\mathbb C)$} and 
\hbox{$\rho_2:\pi_1(S)\to {\rm PSL}(2,\mathbb C)$} are Kleinian surface groups, and
there exists $k$ so that 
and $\ell_{\rho_1}(\alpha)=k\ell_{\rho_2}(\alpha)$ for all $\alpha\in\pi_1(S)$ which are represented
by simple curves on $S$,
then $\rho_1$ is  conjugate to either $\rho_2$ or $\bar\rho_2$.
\end{theorem}

Kim \cite[Thm. 3]{kim-rigid} showed that if $\rho_1$ and $\rho_2$ are irreducible,
non-elementary, nonparabolic representations of a finitely presented group $\Gamma$ into the isometry
group of a rank one symmetric space and there exists $k>0$ such that $\ell_{\rho_1}(\gamma)=k\ell_{\rho_2}(\gamma)$
for all $\gamma\in\Gamma$ (where $\ell_{\rho_i}(\gamma))$
the translation length of
$\rho_i(\gamma)$),  then $k=1$ and
$\rho_1$ and $\rho_2$ are conjugate representations.

\medskip\noindent
{\bf Outline of paper:} In section \ref{complex length} we analyze the complex length spectrum of Kleinian
surface groups with the same simple marked length spectrum, then in section \ref{main proof}, we give the proof of
our main result. In section \ref{acylindrical} we prove Theorem \ref{acyl app}, while in section \ref{pressure} we establish
Theorems \ref{isom} and \ref{ksimple rigidity}

\medskip\noindent
{\bf Acknowledgements:} The authors would like to thank Maxime Wolff for several enlightening conversations
on the length spectra of surface group representations, Jeff Brock and Mike Wolf 
for conversations about the Weil-Petersson metric and Alan Reid for conversations about the character variety.
This material is partially based upon work supported by the National Science Foundation under grant No. 0932078 000 while
the authors were in residence at the Mathematical Sciences Research Institute in Berkeley, CA, during the
Spring 2015 semester.

\section{The complex length spectrum}
{\label{complex length}}

In this section, we investigate the complex length spectra of Kleinian surface groups with the
same simple marked length spectrum.

Given $\alpha \in \pi_1(S)$ and $\rho\in AH(S)$, let $\lambda^2_\rho(\alpha)$
be the square of the largest eigenvalue of $\rho(\alpha)$. Notice that $\lambda^2_\rho(\alpha)$ is well-defined
even though the largest eigenvalue of a matrix in ${\rm PSL}(2,\mathbb C)$ is only well-defined
up to sign. If we choose $\log\lambda^2_\rho(\alpha)$ to have imaginary part in $[0,2\pi)$, then
$\log\lambda^2_\rho(\alpha)$ is the {\em complex length} of $\rho(\alpha)$.

If $\alpha$ is a simple, non-separating closed curve
on $S$, we let $W(\alpha)$ denote the set of all simple, non-separating curves on $S$ which intersect $\alpha$
at most once.  We say that $\rho_1$ and $\rho_2$ have {\em the same marked complex length spectrum} on
$W(\alpha)$ if $\lambda^2_{\rho_1}(\beta) = \lambda^2_{\rho_2}(\beta)$ for all $\beta \in W(\alpha)$. Similarly,
we say that  $\rho_1$ and $\rho_2$ have {\em conjugate marked complex length spectrum} on
$W(\alpha)$ if $\lambda^2_{\rho_1}(\beta) =\overline{\lambda^2_{\rho_2}(\beta)}$ for all $\beta\in W(\alpha)$.

We will show that if two Kleinian surface groups $\rho_1$ and $\rho_2$ have the same simple marked length spectrum,
then,  there exists a simple non-separating curve $\alpha$ on $S$ such that  $\rho_1$ and $\rho_2$ either have
the same or conjugate complex length spectrum on $W(\alpha)$.

\begin{prop}
\label{niceonW}
If $S$ is a closed, connected, orientable surface of genus at least two,
\hbox{$\rho:\pi_1(S)\to {\rm PSL}(2,\mathbb C)$} and 
\hbox{$\rho_2:\pi_1(S)\to {\rm PSL}(2,\mathbb C)$} are Kleinian surface groups with the same
simple marked length spectrum, then
there exists a  simple non-separating curve $\alpha$ on $S$ such that 
$\rho_1(\alpha)$ is hyperbolic and either
\begin{enumerate}
\item
$\rho_1$ and $\rho_2$ have the same marked complex length spectrum on $W(\alpha)$, or
\item 
$\rho_1$ and $\rho_2$ have conjugate marked complex length spectrum on $W(\alpha)$.
\end{enumerate}
\end{prop}

Proposition \ref{niceonW} will be a nearly immediate consequence of three lemmas.
The first lemma shows that for two Kleinian surface groups with the same length spectrum, then the
complex lengths of a simple non-separating curve either agree, differ by complex conjugation, or
differ by sign (and are both real). The second lemma deals with the case where the complex length of
every simple, non-separating curve is real, while the final lemma handles the case where some complex
length is not real. All the proofs revolve around an analysis of the asymptotic behavior of 
complex lengths of curves of the form
$\alpha^n\beta$ where $\alpha$ and $\beta$ intersect exactly once.
We begin by recording computations which will be used repeatedly in the remainder of
the paper.

\subsection{A convenient normalization}
We recall that two elements $\alpha,\beta\in\pi_1(S)$ are {\em coprime} if they share no
common powers. We say that  a representation $\rho:\pi_1(S)\to {\rm PSL}(2,\mathbb C)$ is 
{\em $(\alpha,\beta)$-normalized} if $\alpha,\beta\in\pi_1(S)$ are coprime and $\rho(\alpha)$ 
is hyperbolic and has attracting fixed point $\infty$ and repelling fixed point $0$.
In this case, 
$$\rho(\alpha)= \pm \left( \begin{array}{cc}
\lambda & 0\\
0 &   \lambda^{-1}\end{array}\right) 
$$
where $|\lambda|> 1$, and
$$\rho(\beta)= \pm \left( \begin{array}{cc}
a & b\\
c &  d \end{array}\right).$$
where $ad-bc= 1.$
Notice that the matrix representations of elements of ${\rm PSL}(2,\mathbb C)$ are only well-defined
up to multiplication by $\pm I$, but many related quantities like the square of the trace, the product
of any two co-efficients, and the modulus of the eigenvalue of maximal modulus are well-defined.

\begin{lemma}
\label{basic comp}
Suppose that $S$ is a closed, connected, orientable surface of genus at least two and
\hbox{$\rho:\pi_1(S)\to {\rm PSL}(2,\mathbb C)$} is an $(\alpha,\beta)$-normalized Kleinian surface group.
In the above notation,
$$\rho(\alpha^n\beta)= \pm \left( \begin{array}{cc}
\lambda^na & \lambda^{n}b\\
\lambda^{-n}c &  \lambda^{-n}d \end{array}\right)$$
and all the  matrix coefficients of $\rho(\beta)$ are non-zero.
Moreover, if $\mu(n)$ is the modulus of the eigenvalue of $\rho(\alpha^n\beta)$ with largest modulus,
then
$$\log\mu(n)= n\log|\lambda| + \log|a| +   \Re\left( \lambda^{-2n}\frac{bc}{a^{2}}\right) + O(|\lambda|^{-4n}).$$
\end{lemma}

\begin{proof}
The first claim follows from a simple computation. If any of the coefficients of $\rho(\beta)$ are 0,
then $\rho(\beta)$ takes some fixed point of $\rho(\alpha)$ to a fixed point of $\rho(\alpha)$, e.g.
if $a=0$, then $\rho(\beta)(\infty)=0$. This would imply that $\rho(\beta\alpha\beta^{-1})$ shares a fixed
point with $\rho(\alpha)$. Since $\rho(\pi_1(S))$ is discrete, this would imply that there is an element
which is a power of both $\rho (\alpha)$ and $\rho(\beta\alpha\beta^{-1})$ which would contradict
the facts that $\rho$ is faithful and the subgroup of $\pi_1(S)$ generated by $\alpha$ and $\beta$ is free of rank two.

The eigenvalues of $\rho(\alpha^n\beta)$, which are only well-defined up to sign, are then given by
\begin{eqnarray*}
\pm \left(  \frac{ (\lambda^{n}a +  \lambda^{-n}d) \pm \sqrt{( \lambda^{n}a +  \lambda^{-n}d)^{2}-4}}{2} \right)
\end{eqnarray*}
So, since $|\lambda|>1$, for all large enough $n$, one may use the Taylor expansion for $\sqrt{1+x}$ to conclude that
they have the form
\begin{eqnarray*}
\pm \left(\lambda^{n}a\left(1 +  \lambda^{-2n}\left(\frac{ad-1}{a^{2}}\right) + O(\lambda^{-4n})\right)\right)
\end{eqnarray*}
Therefore, since $ad-bc=1$,
\begin{eqnarray*}
\log\mu(n) &=& n\log|\lambda| + \log|a| + \log\left|  1 +  \lambda^{-2n}\frac{bc}{a^{2}} + O(\lambda^{-4n})\right|. 
\end{eqnarray*}
We then use the expansion of $\log|1 + z|$ about $z=0$ given by
$$\log|1 + z| = \frac{1}{2}\log(|1+z|^2) = \frac{1}{2}\log(1+ 2\Re(z) + |z|^2) = \Re(z) + O(|z|^2)$$
to show that
\begin{eqnarray*}
\log\mu(n) &=& n\log|\lambda| + \log|a| +   \Re\left( \lambda^{-2n}\frac{bc}{a^{2}}\right) + O(|\lambda|^{-4n}).
\end{eqnarray*}
\end{proof}

\subsection{Basic relationships between complex lengths}
Our first lemma shows that if two Kleinian surface groups have the same simple marked length spectrum,
then the complex lengths of any simple, non-separating curve either agree or differ by either complex conjugation or
sign.

\begin{lemma}
\label{tri}
If $\rho_1$ and $\rho_2$ in $AH(S)$ have the same simple marked length spectrum
and $\alpha$ is a simple non-separating curve on $S$, then either
\begin{enumerate}
\item
$\lambda^2_{\rho_1}(\alpha)  = \lambda^2_{\rho_2}(\alpha)$, 
\item
$\lambda^2_{\rho_1}(\alpha)= \overline{\lambda^2_{\rho_2}(\alpha)}$, or
\item
$\lambda^2_{\rho_1}(\alpha) = -\lambda^2_{\rho_2}(\alpha)$ and $\lambda^2_{\rho_1}(\alpha)$ is real.
\end{enumerate}
\end{lemma}

\begin{proof}
If $\rho_1(\alpha)$ is parabolic, then $\rho_2(\alpha)$ is parabolic (since $\ell_{\rho_2}(\alpha)=\ell_{\rho_1}(\alpha)=0$).
In this case,  $\lambda^2_{\rho_1}(\alpha)=\lambda^2_{\rho_2}(\alpha)=1$. So we may assume that $\rho(\alpha)$
is hyperbolic.

Let $\beta$ be a curve intersecting $\alpha$ exactly once.
We may assume that both $\rho_1$ and $\rho_2$ are $(\alpha,\beta)$-normalized, so
$$\rho_i(\alpha) =\pm  \left( \begin{array}{cc}
\lambda_i & 0\\
0 &   \lambda_i^{-1}\end{array}\right) 
$$ where $|\lambda_i|> 1$, and
$$\rho_i(\beta) = \pm \left( \begin{array}{cc}
a_i & b_i\\
c_i &  d_i \end{array}\right).$$
where $a_id_i-b_ic_i = 1.$

Since,  $\rho_1$ and $\rho_2$ have the same simple marked length spectrum and $|\lambda_1| = |\lambda_2|$,
Lemma \ref{basic comp} implies that  
$$\log |a_1|+ \Re\left( \lambda_1^{-2n}\frac{b_1c_1}{a_1^{2}}\right) + O(|\lambda_1|^{-4n}) = \log |a_2|+ \Re\left( \lambda_2^{-2n}\frac{b_2c_2}{a_2^{2}}\right) + O(|\lambda_2|^{-4n})$$
for all $n$.
Taking limits as $n\to\infty$, we see that $\log |a_1|=\log |a_2|$, so
$$\Re\left( \lambda_1^{-2n}\frac{b_1c_1}{a_1^{2}}\right) + O(|\lambda_1|^{-4n}) = \Re\left( \lambda_2^{-2n}\frac{b_2c_2}{a_2^{2}}\right) + O(|\lambda_2|^{-4n})$$
for all $n$. Therefore, after multiplying both sides by $|\lambda_1|^{2n}=|\lambda_2|^{2n}$, we see that
$$\lim_{n \rightarrow \infty}\Re\left(u_1^{n}v_1 - u_2^{n}v_2\right) =0$$
where
$$u_i = \left(\frac{\lambda_{i}}{|\lambda_{i}|}\right)^{-2}  \quad\textrm{and}\quad v_i = \frac{b_ic_i}{a_i^2} \neq 0.$$

Lemma \ref{tri} is then an immediate consequence of the following elementary lemma.
\end{proof}

\begin{lemma}{\label{rotate}}
If $u_1, u_2 \in {\mathbb S}^1$, $v_1, v_2 \in {\mathbb C}-\{0\}$ and
$$\lim_{n \rightarrow \infty}\Re\left(u_1^{n}v_1 - u_2^{n}v_2\right) =0,$$
then either
\begin{enumerate}
\item
$ u_1 = u_2$,
\item
$ u_1 = \overline{u}_2$, or 
\item
$u_1 = - u_2 = \pm 1$.
\end{enumerate}
\end{lemma}

\begin{proof}

We choose $\theta_i$  so that
$$u_i=e^{i\theta_i} $$

If  $s = \{n_k\}_{k=1}^\infty$  is a strictly increasing sequence of integers, let 
$S_i(s)$ be the set of accumulation points of  $\{\Re(u_i^{n_{k}}v_i)\}$.
Then, by assumption, 
$S_1(s) = S_2(s)$ for any sequence $s$. 

If $\theta_i$ is an irrational multiple of $2\pi$,  then $S_i(\N)$ is the interval $[-|v_i|,|v_i|]$.  
If  $\theta_i$ is a rational multiple of $2\pi$ 
then $S_i(\N)$ is finite. Therefore either (a) both $\theta_1$ and $\theta_2$  are irrational with $|v_1| = |v_2|$ or  (b) both $\theta_1$ and $\theta_2$ are rational multiples of $2\pi$.  We handle these two cases separately.

\medskip\noindent
{\bf Case (a):} {\bf Both $\theta_1$ and $\theta_2$ are  irrational multiples of $2\pi$
and $|v_1| = |v_2|$:}
Since $\theta_1$ is an irrational multiple of $2\pi$, there is a sequence $\{n_k\}$ such that 
$\lim_{ k\rightarrow \infty} e^{in_k\theta_1} = \frac{\overline{v_1}}{|v_1|}$.
Therefore,
$$|v_1| = \lim_{ k\rightarrow \infty} \Re(e^{in_k\theta_1}v_1) = \lim_{k\rightarrow \infty} \Re(e^{in_k\theta_2}v_2) = |v_2|,$$
so $\lim_{k\rightarrow \infty} e^{in_k\theta_2} = \frac{\overline{v}_2}{|v_2|}$.
If $\{m_k\} = \{n_k+1\}$, then
$$|v_1|\cos{\theta_1}  =\lim_{ k\rightarrow \infty} \Re(e^{im_k\theta_1}v_1) =  \lim_{ k\rightarrow \infty} \Re(e^{im_k\theta_2}v_2) 
= |v_2|\cos{\theta_2}.$$
Since $|v_1|=|v_2|\ne 0$,  it follows that $\theta_1 = \pm \theta_2$, so either $u_1 = u_2$ or $u_1 = \overline{u}_2$
and we are either in case (1) or in case (2).

\medskip\noindent
{\bf Case II:} {\bf Both $\theta_1$ and $\theta_2$ are  rational multiples of $2\pi$:} 
Let $\theta_i = 2\pi p_i/q_i$ where $0\le p_i < q_i$ and $p_i$ and $q_i$ are relatively prime (and $q_i=1$ if $p_i=0$).

If  $r\in\Z$ and $s_r=\{r + kq_1q_2\}$, then
$$S_1(s_r) = \{\Re(u_1^rv_1) \} = S_2(s_r) = \{\Re(u_2^rv_2) \},$$
so
$$\Re(u_1^rv_1) = \Re(u_2^rv_2)\qquad \mbox{ for all $r\in\Z$.}$$

If $\Re(u_1^rv_1) = \Re(u_2^rv_2) = 0$ for all $r\in\Z$, then  $u_1 =\pm u_2 = \pm 1$, and we are in either case (1) or   case (3). 

If $\Re(u_1^rv_1) = \Re(u_2^rv_2) \neq 0$ for some $r$, then
$$2\cos(\theta_1)\Re(u_1^rv_1) = \Re(u_1^{r+1}v_1) + \Re(u_1^{r-1}v_1) =  \Re(u_2^{r+1}v_2) + \Re(u_2^{r-1}v_2) = 2\cos(\theta_2)\Re(u_2^rv_2).$$
Since $\Re(u_1^rv_1) = \Re(u_2^rv_2) \neq 0$, this implies that
$\cos(\theta_1) = \cos(\theta_2)$, so $\theta_1 = \pm\theta_2$. Therefore, either $u_1 = u_2$ or $u_1 = \overline{u}_2$
and we are in either case (1) or case (2).

This completes the proof, since in all situations we have seen that either case (1), (2) or (3) occurs.
\end{proof}

\subsection{When the simple non-separating complex length spectrum is totally real}
We use a similar analysis to show that  if the complex lengths of every simple, non-separating
curve is real for two Kleinian surface groups with the same simple marked length spectrum,
then the complex lengths agree for every simple, non-separating curve.

\begin{lemma}
\label{if not real}
If $S$ is a closed, connected, orientable surface of genus at least two and
\hbox{$\rho_1:\pi_1(S)\to {\rm PSL}(2,\mathbb C)$} and 
\hbox{$\rho_2:\pi_1(S)\to {\rm PSL}(2,\mathbb C)$} are Kleinian surface groups with same
simple marked length spectrum,  then either
\begin{enumerate} 
\item 
there exists a simple non-separating curve $\gamma$  on $S$ such that $\lambda^2_{\rho_1}(\gamma) \not\in \R$, or
\item 
$\lambda^2_{\rho_1}(\gamma) = \lambda^2_{\rho_2}(\gamma)\in\R$ whenever
$\gamma$  is a simple non-separating curve on $S$.
\end{enumerate}
\end{lemma}

\begin{proof}
Suppose that  1) does not hold, so $\lambda^2_{\rho_1}(\gamma) \in \R$ 
whenever $\gamma$ is a simple non-separating curve on $S$. Lemma \ref{tri}  then implies that 
$\lambda^2_{\rho_1}(\gamma) = \pm \lambda^2_{\rho_2}(\gamma)$ whenever 
$\gamma$ is a simple non-separating curve on $S$

Suppose that  there is a simple non-separating curve $\alpha$ such that 
$\lambda^2_{\rho_1}(\alpha) = - \lambda^2_{\rho_2}(\alpha)$. Notice that  if $\rho_1(\alpha)$ is parabolic, then
$\ell_{\rho_1}(\alpha)=0=\ell_{\rho_2}(\alpha)$, so $\lambda^2_{\rho_1}(\alpha) = 1= \lambda^2_{\rho_2}(\alpha)$.
Therefore, $\rho_1(\alpha)$, and hence $\rho_2(\alpha)$, must be hyperbolic

We  choose a  simple non-separating curve $\beta$ intersecting $\alpha$ exactly once.
We adapt the normalization and notation of Lemma \ref{tri}.
Lemma \ref{basic comp} implies that 
$$ t_i(n)={\rm Tr}^2(\rho_i(\alpha^n\beta)) = \lambda_i^{2n}a_i^2 + 2a_id_i + \lambda_i^{-2n}d_i^2=
\lambda^2_{\rho_i}(\alpha^n\beta)+2+\lambda^{-2}_{\rho_i}(\alpha^n\beta).$$ 
(Notice that  the trace ${\rm Tr}(\rho_i(\alpha^n\beta))$ of $\rho_i(\alpha^n\beta)$ is well-defined up to sign,
so ${\rm Tr}^2(\rho_i(\alpha^n\beta))$ is well-defined.)
Since $\lambda^2_{\rho_1}(\alpha^n\beta)=\pm\lambda^2_{\rho_2}(\alpha^n\beta)$ for all $n$, by assumption,
either (i) $t_1(n) = t_2(n)$ or  (ii) $t_1(n) = 4-t_2(n)$ for all $n$ and $\lambda_1^2=-\lambda_2^2$.

The proof divides into two cases.

\medskip\noindent
{\bf Case I: There is an infinite sequence $\{n_k\}$ of even integers so that $t_1(n_k) = t_2(n_k)$:}
Dividing by $\lambda_1^{2n_k}=\lambda_2^{2n_k}$ and taking limits we
see that 
$$a_1^2=\lim_{k\to\infty} a_1^2+\frac{2a_1d_1}{\lambda_1^{2n_k}} +\frac{d_1^2}{\lambda_1^{4n_k}}=
\lim_{k\to\infty} a_2^2+\frac{2a_2d_2}{\lambda_2^{2n_k}} +\frac{d_2^2}{\lambda_2^{4n_k}}=a_2^2.$$
It follows that
$$2a_1d_1+\frac{d_1^2}{\lambda_1^{2n_k}}=2a_2d_2+\frac{d_2^2}{\lambda_2^{2n_k}}$$ 
for all $n_k$, so, after again taking limits, we see that 
$$a_1d_1=a_2d_2.$$
Then, by considering the final term, we see that
$$d_1^2=d_2^2.$$

If there exists an infinite sequence $\{m_j\}$ of odd integers so that 
$t_1(m_j) = t_2(m_j)$ for all $m_j$, then
$$ \lambda_1^{2m_j}a_1^2 + 2a_1d_1 + \lambda_1^{-2m_j}d_1^2  = 
 -\lambda_1^{2m_j}a_1^2 + 2a_1d_1- \lambda_1^{-2m_j}d_1^2.$$
Then we may divide  each side by $\lambda_1^{2m_j}$ and pass to a limit to conclude that
$a_1^2 = -a_1^2$. This would imply that  $a_1 = 0$, which would contradict Lemma \ref{basic comp}.

On the other hand, if there exists an infinite sequence $\{m_j\}$ of odd integers so that 
$t_1(m_j) = 4-t_2(m_j)$ for all $m_j$, then
$$ \lambda_1^{2m_j}a_1^2 + 2a_1d_1 + \lambda_1^{-2m_j}d_1^2  = 
4-( -\lambda_1^{2m_j}a_1^2 + 2a_1d_1- \lambda_1^{-2m_j}d_1^2),$$
so
$$2a_1d_1 = 4-2a_1d.$$ 
Therefore, $a_1d_1 = 1$,  which implies  that $b_1c_1 = 0$, 
so  either $b_1 = 0$ or $c_1 = 0$, which again contradicts Lemma \ref{basic comp}.

\medskip\noindent
{\bf Case II: There is an infinite sequence $\{n_k\}$ of even integers so that
\hbox{$t_1(n_k) = 4-t_2(n_k)$}:}
We then argue, as in Case I, to show that
$$a_1^2 = -a_2^2\quad\textrm{and}\quad 2a_1d_1 = 4-2a_2d_2 \quad\textrm{and}\quad d_1^2 = -d_2^2.$$

If there exists an infinite sequence $\{m_j\}$ of odd integers so that 
$t_1(m_j) = t_2(m_j)$ for all $m_j$, then
$$ \lambda_1^{2m_j}a_1^2 +2a_1d_1 + \lambda_1^{-2m_j}d_1^2  = 
\lambda_1^{2m_j}a_1^2 +(4- 2a_1d_1)+ \lambda_1^{-2m_j}d_1^2.$$
So, $a_1d_1=1$, again giving a contradiction.

On the other hand, if there exists an infinite sequence $\{m_j\}$ of odd integers so that 
$t_1(m_j) = 4-t_2(m_j)$ for all $m_j$, then
$$ \lambda_1^{2m_j}a_1^2 +2a_1d_1 + \lambda_1^{-2m_j}d_1^2  =4-(\lambda_1^{2m_j}a_1^2 +(4- 2a_1d_1)+ \lambda_1^{-2m_j}d_1^2).$$
Dividing both sides by $\lambda_1^{2m_j}$ and passing to a limit, we conclude that
$a_1^2 = -a_1^2$, which is again  a contradiction.

Therefore, neither Case I or Case II can occur, so case  (2) must hold.
 \end{proof}

\subsection{When the  complex length is not  always totally real}
We now show that if $\lambda_{\rho_1}^2(\alpha)$ is not real, for some simple non-separating curve
$\alpha$, then $\lambda_{\rho_1}^2(\beta)$ and $\lambda^2_{\rho_2}(\beta)$ either
agree for all $\beta\in W(\alpha)$, or differ by complex conjugation for all $\beta\in W(\alpha)$.

\begin{lemma}
\label{unreal case}
Suppose that $S$ is a closed, connected, orientable surface of genus at least two and
\hbox{$\rho_1:\pi_1(S)\to {\rm PSL}(2,\mathbb C)$} and 
\hbox{$\rho_2:\pi_1(S)\to {\rm PSL}(2,\mathbb C)$} are Kleinian surface groups with same
simple marked length spectrum.  If
$\alpha$ is simple non-separating curve on $S$ such that $\lambda^2_{\rho_1}(\alpha) \not\in \R$, then either
\begin{enumerate}
\item
$\rho_1$ and $\rho_2$ have the same marked complex length spectrum on $W(\alpha)$, or
\item 
$\rho_1$ and $\rho_2$ have conjugate marked complex length spectrum on $W(\alpha)$.
\end{enumerate}
\end{lemma}

\begin{proof}
Lemma \ref{tri} implies that either $\lambda^2_{\rho_1}(\alpha) = \lambda^2_{\rho_2}(\alpha)$ or 
$\lambda^2_{\rho_1}(\alpha) = \overline{\lambda^2_{\rho_2}(\alpha)}$. 
If $\lambda^2_{\rho_1}(\alpha) = \overline{\lambda^2_{\rho_2}(\alpha)}$, then we consider the
representation $\bar\rho_2$.
In this case, $\lambda^2_{\rho_2}(\gamma)=\overline{\lambda_{\bar\rho_2}^2(\gamma)}$ for all $\gamma\in\pi_1(S)$.
In particular, $\lambda^2_{\rho_1}(\alpha) = \lambda^2_{\overline{\rho}_2}(\alpha)$. 
Therefore,  it suffices to prove that  $\rho_1$ and $\rho_2$ have the same marked complex length spectrum on $W(\alpha)$
whenever $\lambda^2 = \lambda^2_{\rho_1}(\alpha) = \lambda^2_{\rho_2}(\alpha)$ and $\rho_1$ and $\rho_2$ have
the same simple marked length spectrum.
 
First, suppose that $\beta$ is a simple non-separating curve on $S$ which intersects $\alpha$ once. 
We adopt the normalization and notation of Lemmas \ref{tri} and \ref{if not real}, so
$$ t_i(n)=\Tr^2(\rho_i(\alpha^n\beta)) = \lambda^{2n}a_i^2 + 2a_id_i + \lambda^{-2n}d_i^2 $$
Lemma \ref{tri} implies that for any $n$, either (i) $t_1(n)=t_2(n)$, (ii) $t_1(n)=\overline{t_2(n)}$ or (iii)
$t_1(n)=4-t_2(n)$ and $t_1(n)$ is real.

If there is an infinite set of values of $n$ such that $t_1(n) = t_2(n)$,  then, by taking limits, we see that 
$$a_1^2 =a_2^2, \quad a_1d_1 = a_2d_2\quad\textrm{and}\quad d_1^2 = d_2^2.$$
It follows then that $t_1(n) = t_2(n)$ for all $n$.
Moreover, since $a_i$ and $d_i$ are non-zero, either $a_1=a_2$ and $d_1=d_2$ or $a_1=-a_2$ and $d_1=-d_2$, so
${\rm Tr}^2(\rho_1(\beta))={\rm Tr}^2(\rho_2(\beta))$ which implies that 
$\lambda^2_{\rho_1}(\beta) = \lambda^2_{\rho_2}(\beta)$.

If  there is an infinite set of values $n$ such that $t_1(n) = 4-t_2(n)$ with $t_i(n)$ real,  then
taking limits we have
$$a_1^2 = -a_2^2, \quad 2a_1d_1 = 4 - 2a_2d_2\quad\textrm{and}\quad d_1^2 = -d_2^2.$$
It follows that $t_1(n) =4 -t_2(n)$ for all $n$, so by Lemma \ref{tri}, $t_1(n)$ is real for all $n$.
Since $\Im(t_i(n)) = 0$ for all $n$,
$$\lim_{n\to\infty} \frac{\Im(t_i(n))}{|\lambda|^{2n}} =\lim_{n\to\infty} \Im\left(\left(\frac{\lambda^2}{|\lambda|^2}\right)^{n}a_i^2\right)= 0.$$
Since $a_i^2 \neq 0$, by Lemma \ref{basic comp}, this can only happen if 
$\frac{\lambda^2}{|\lambda|^2}= \pm 1$. Thus, $\lambda^2=\lambda_{\rho_1}(\alpha^2)$ is real, contradicting our assumption. 

Finally, if $t_1(n) = \overline{t_2(n)}$  for all but finitely many values of $n$, we may
divide the resulting equation by $\overline{\lambda^{2n}}$ and
take a limit, to see that  
$$\lim_{n \rightarrow \infty} \left(\frac{\lambda^2}{\overline{\lambda^2}}\right)^{n}a_1^2 = \overline{a_2^2} .$$
Since $\lambda^2 \neq \overline{\lambda^2}$, the limit does not exist unless $a_1 = 0$,
which again contradicts Lemma \ref{basic comp}.

Therefore, if $\beta\in W(\alpha)$ intersects $\alpha$ once, then
$\lambda^2_{\rho_1}(\beta) = \lambda^2_{\rho_2}(\beta)$.

Now suppose that $\beta$ is a simple non-separating curve on $S$ which does not intersect $\alpha$.
We choose $\beta'$ to be a simple non-separating curve  intersecting both $\alpha$ and $\beta$ once.

For all $n$, $\alpha^n\beta'\in W(\alpha)$ and  intersects $\alpha$ once.
By the first part of the argument,
$$\lambda^2_{\rho_1}(\alpha^n\beta')=\lambda^2_{\rho_2}(\alpha^n\beta')$$
for all $n$. 
If there exists $n_0$ so that $\lambda^2_{\rho_1}(\alpha^{n_0}\beta')$
is not real, then since $\beta\in W(\alpha^{n_0}\beta')$ and intersects $\alpha^{n_0}\beta'$ exactly once,
we may apply the above argument to show that 
$\lambda^2_{\rho_1}(\beta) = \lambda^2_{\rho_2}(\beta)$.

It remains to consider the case that $\lambda^2_{\rho_1}(\alpha^{n}\beta')$ is real for all $n$. 
Suppose that
$$\rho_i(\beta') = \pm \left( \begin{array}{cc}
a_i' & b_i'\\
c_i' &  d_i' \end{array}\right).$$
Again, by Lemma \ref{basic comp} all the matrix coefficients must be non-zero.
Since $\Im(\lambda^2_{\rho_1}(\alpha^n\beta'))=0$ for all $n$, $\Im(\Tr^2(\rho_1(\alpha^n\beta')))=0$ for all $n$,
so, after dividing the resulting equation by $|\lambda|^{2n}$, for all $n$, and passing to the limit we see that
$$\lim_{n\rightarrow \infty}  
\Im\left(\left(\frac{\lambda^2}{|\lambda|^2}\right)^{n}a_1'^2\right)= 0.$$
Since $\frac{\lambda^2}{|\lambda|^2}\notin\R$,
this implies that $a_1' = 0$, which is again a contradiction. 
Therefore, if $\beta\in W(\alpha)$ does not intersect $\alpha$, then 
$\lambda^2_{\rho_1}(\beta) = \lambda^2_{\rho_2}(\beta)$
 which completes the proof.
\end{proof}

\subsection{Assembly}
We can now easily assemble the proof of Proposition \ref{niceonW}.

\medskip\noindent
{\em Proof of Proposition \ref{niceonW}:}
If there exists a simple, non-separating curve $\alpha$ on $S$ so that $\lambda^2_{\rho_1}(\alpha)$ is not
real, then Proposition \ref{niceonW} follows immediately from Lemma \ref{unreal case}.
If $\lambda^2_{\rho_1}(\alpha)$ is real for every simple, non-separating curve $\alpha$ on $S$, then
Lemma \ref{if not real} implies that $\rho_1$ and $\rho_2$ have the same marked complex length spectrum on $W(\alpha)$
for any non-separating simple closed curve $\alpha$.
A result of Sullivan \cite{sullivan-finite} implies that there 
are only finitely many simple curves $\gamma$ on $S$ so that $\rho_1(\gamma)$ is parabolic,
so we may always choose $\alpha$ so that $\rho_1(\alpha)$ is hyperbolic.
\qed

\medskip\noindent
{\bf Remark:} An examination of the proofs  reveals that Proposition
\ref{niceonW} holds  whenever  the length spectra of $\rho_1$ 
and $\rho_2$ agree on all simple, non-separating curves.

\section{Simple Marked Length Spectrum Rigidity}
\label{main proof}

We are now ready for the proof of our main result. 

\medskip\noindent
{\bf Theorem \ref{main}:} {\rm (Simple length rigidity for Kleinian surface groups)} 
{\em If $S$ is a closed, connected,
orientable surface of genus at least two, and $\rho_1:\pi_1(S)\to {\rm PSL}(2,\mathbb C)$ and 
$\rho_2:\pi_1(S)\to {\rm PSL}(2,\mathbb C)$ are Kleinian surface groups with the same
simple marked length spectrum, then $\rho_1$ is conjugate to either  $\rho_2$ or $\bar\rho_2$.}

\medskip

We begin with a brief sketch of the proof. It follows from Lemma \ref{niceonW}  that, perhaps
after replacing $\rho_2$ with a complex conjugate representation,
there exists a simple, non-separating curve $\alpha$ so that $\rho_1(\alpha)$ is hyperbolic and 
$\rho_1$ and $\rho_2$  have the same marked complex length spectrum on $W(\alpha)$.
We then lift $\rho_1$ and $\rho_2$ to representations into ${\rm SL}(2,\mathbb C)$ which
have the same trace on a standard set of generators $\{\alpha_1,\beta_1,\ldots,\alpha_g,\beta_g\}$
where $\alpha=\alpha_1$, see Lemma \ref{lifts exist}. An analysis of the asymptotic behavior of the traces of 
$\alpha_j^n\beta$ allows us to conclude that the restriction of the lifts to any subgroup of the form 
$G_j=< \alpha_j,\beta_j>$ are conjugate, see Lemma \ref{equal on Gi}. A more intricate analysis of the same type
is then applied to show that if we conjugate the
lifts to agree on $G_j$, then, for any $k$, they either agree on $G_k$ or differ by conjugation by a lift of the rotation of order two
in the axis of the image of the commutator of $[\alpha_j,\beta_j]$, see Lemmas \ref{generators} and \ref{step one}.
The proof is then easily completed when the genus is greater than two, see Lemma \ref{general case}, but
a separate analysis is required when the genus is two, see Lemma \ref{genus two case}.

\medskip\noindent
{\em Proof of Theorem \ref{main}:}
Proposition \ref{niceonW} implies that there exists a 
a simple non-separating curve $\alpha$ such that
$\rho_1(\alpha)$ is hyperbolic and $\rho_1$ and $\rho_2$ have the same marked complex length spectrum on $W(\alpha)$.
If $\rho_1$ and $\rho_2$ have conjugate complex marked length spectrum on $W(\alpha)$, then 
$\rho_1$ and $\bar\rho_2$ have the same marked complex length spectrum on $W(\alpha)$.
Therefore, we may assume that $\rho_1$ and $\rho_2$ have the same marked complex length spectrum on
$W(\alpha)$.

We begin by choosing lifts whose traces agree on a standard set of generators which includes $\alpha$.
We will call $\mathcal S=\{\alpha_1,\beta_1,\ldots,\alpha_g,\beta_g\}$ a {\em standard set of generators} for $\pi_1(S)$ if
each $\alpha_j$ and $\beta_j$ is non-separating, $\pi_1(S) = <\alpha_j,\beta_j\ | \ \prod_{j=1}^g[\alpha_i,\beta_i] = id>$ and $i(\alpha_j, \beta_j) = 1$  for all $j$ and 
if $j \neq k$ then 
$$i(\alpha_j,\alpha_k) = i(\beta_j, \beta_k) = i(\alpha_j,\beta_k) =  0,$$
see Figure \ref{generators_fig}.  We say that two lifts $\tilde\rho_1$ and $\tilde\rho_2$ 
of $\rho_1$ and $\rho_2$ are {\em trace normalized} with respect to $\mathcal S$ if
\begin{enumerate}
\item
$\rho_1(\delta)$ is hyperbolic  for all $\delta\in\mathcal S$, and
\item
$\Tr(\tilde\rho_1(\delta)) =  \Tr(\tilde\rho_2(\delta))$ for all $\delta \in \mathcal S$. 
\end{enumerate}

 \begin{figure}[htbp] 
    \centering
    \includegraphics[width=4in]{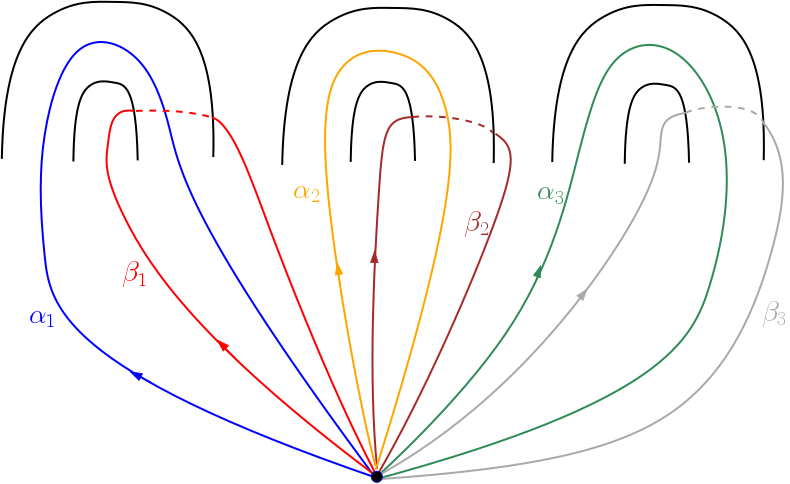} 
    \caption{Generators of $\pi_1(S,p)$}
    \label{generators_fig}
 \end{figure}

\begin{lemma}
\label{lifts exist}
Suppose that $\rho_1:\pi_1(S)\to {\rm PSL}(2,\mathbb C)$ and $\rho_2:\pi_1(S)\to {\rm PSL}(2,\mathbb C)$
are Kleinian surface groups with the same marked complex length spectrum on $W(\alpha)$ for some
simple, non-separating curve $\alpha$. If $\rho_1(\alpha)$ is hyperbolic,
then there exists a standard set of generators $\mathcal S$, so that $\alpha_1=\alpha$,
and lifts $\tilde\rho_1:\pi_1(S)\to {\rm SL}(2,\mathbb C)$ and $\tilde\rho_2:\pi_1(S)\to {\rm SL}(2,\mathbb C)$
of $\rho_1$ and $\rho_2$ which are trace normalized with respect to $\mathcal S$.
\end{lemma}

\begin{proof}
Let $\alpha_1=\alpha$. Choose a simple non-separating curve $\beta$ which intersects $\alpha_1$ exactly
once so that $\rho_1(\beta_1)$ is hyperbolic. (We may do so, since, by a result of Sullivan \cite{sullivan-finite},
there are only finitely many simple curves $\gamma$ such that $\rho_1(\gamma)$ is parabolic 
and there are infinitely many possibilities for $\beta_1$.)
Extend $\{\alpha_1,\beta_1\}$ to a standard set of generators $\{\alpha_1,\beta_1,\ldots,\alpha_g,\beta_g\}$.
We may assume that $\rho_1(\alpha_j)$ is hyperbolic for all $j\ge 2$, by
replacing $\alpha_j$ by $\alpha_j\beta_j^n$ for some $n$ if necessary.
We may then assume that $\rho_1(\beta_j)$ is hyperbolic
for all $j\ge 2$ by replacing $\beta_j$ by $\beta_j\alpha_j^n$ for some $n$ if necessary. 
Notice that ${\mathcal S} \subseteq W(\alpha)$.

Since each $\rho_i$ is discrete and faithful, each $\rho_i$ lifts to a representation \hbox{$\rho_i':\pi_1(S)\to{\rm SL}(2,\mathbb C)$}
(see Culler \cite{culler-lift} or Kra \cite{kra-lift}).
Let
$$\tilde{\rho}_2(\delta) = \left\{\begin{array}{cc}
\rho_2'(\delta) &\qquad \mbox{if } \Tr(\rho'_1(\delta)) =  \Tr(\rho'_2(\delta))\\
-\rho_2'(\delta) &\qquad  \mbox{if }  \Tr(\rho'_1(\delta)) =  -\Tr(\rho'_2(\delta))\end{array}
\right. $$
for all $\delta\in\mathcal S$.
Notice that $\tilde\rho_i(\delta)$ is a lift of $\rho_i(\delta)$ for all $\delta\in\mathcal{S}$
and that $\tilde\rho_i( \prod_{j=1}^g[\alpha_j,\beta_j] )=I$, since $\rho_i'( \prod_{j=1}^g[\alpha_j,\beta_j])=I$. 
Therefore, $\tilde\rho_1$ and $\tilde\rho_2$ are lifts of $\rho_1$ and $\rho_2$ which are trace normalized with
respect to $\mathcal S$. 
\end{proof}

We next show that the trace normalized lifts are conjugate on the subgroups  \hbox{$G_j = <\alpha_j, \beta_j>$.}

\begin{lemma} 
Suppose that $\tilde\rho_1:\pi_1(S)\to {\rm SL}(2,\mathbb C)$ and $\tilde\rho_2:\pi_1(S)\to {\rm SL}(2,\mathbb C)$
are trace normalized lifts, with respect to a standard generating set $\mathcal S$, of Kleinian
surface groups with the same marked complex length spectrum on $W(\alpha_1)$.
If $j\in\{1,\ldots,g\}$, then there exists \hbox{$K_j\in {\rm SL}(2,\mathbb C)$} such that 
$\tilde\rho_2|_{G_j}=(K_j\tilde\rho_1K_j^{-1})|_{G_j}$.
In particular, if $\gamma\in G_j$, then $\Tr(\tilde\rho_1(\gamma)) = \Tr(\tilde\rho_2(\gamma))$.
\label{equal on Gi}
\end{lemma}

\begin{proof}
Fix $j$ for the remainder of the proof of the lemma and assume that $\tilde\rho_1$ and $\tilde\rho_2$
are $(\alpha_j,\beta_j)$-normalized, so
$$\tilde\rho_1(\alpha_j)  = \tilde\rho_2(\alpha_j) = \left( \begin{array}{cc}
\lambda & 0\\
0 &   \lambda^{-1}\end{array}\right) 
$$
where $|\lambda|>1$, and
$$\tilde\rho_i(\beta_j) = \left( \begin{array}{cc}
a_i & b_i\\
c_i &  d_i \end{array}\right).$$
where $a_id_i-b_ic_i = 1$ and all the co-efficients are non-zero, by Lemma \ref{basic comp}.
Since $\Tr(\tilde\rho_1(\beta_j)) = \Tr(\tilde\rho_2(\beta_j))$, 
$$a_1 + d_1 = a_2+ d_2.$$

The curve $\alpha_j^n\beta_j \in W(\alpha_1)$, for all $n$, since
it is non-separating, simple and disjoint from $\alpha$ if $j >1$ and intersects $\alpha$
exactly once if $j=1$.
Therefore, $\Tr(\tilde\rho_1(\alpha_j^n\beta_j)) = \pm \Tr(\tilde\rho_2(\alpha_j^n\beta_j))$,
which gives the equation
$$\lambda^n a_1 + \lambda^{-n} d_1 = \pm(\lambda^n a_2 + \lambda^{-n} d_2).$$

If $\lambda^n a_1 + \lambda^{-n} d_1 = -(\lambda^n a_2 + \lambda^{-n} d_2)$ for infinitely many values of $n$,
then we see, by dividing by $\lambda^{n}$ and passing to a limit, that
$a_1 = -a_2$. It then follows that $d_1 = -d_2$. Thus, $a_1 + d_1 = -(a_2+ d_2)$ 
which contradicts the fact that the traces of $\tilde\rho_i$ agree (and are non-zero) on elements of $\mathcal S$.

Therefore, there are an infinitely many values  of $n$ where 
 $\lambda^n a_1 + \lambda^{-n} d_1 = \lambda^n a_2 + \lambda^{-n} d_2$
Taking limits again, we see that $a_1 = a_2$ and $d_1 = d_2$.
Therefore,
$$b_1c_1 = 1-a_1d_1 = 1-a_2d_2 = b_2c_2.$$
Choose $u \in \CC$ such that $u^2 = b_2/b_1 = c_1/c_2$. Then $u^2b_1 = b_2$ and $u^{-2}c_1 = c_2$. Let
$$K_j = \left( \begin{array}{cc}
u & 0\\
0 &  u^{-1} \end{array}\right).$$
Notice that $K_j\tilde\rho_1(\alpha_j)K_j^{-1}=\tilde\rho_1(\alpha_j)$
and that
$$K_j\tilde\rho_1(\beta_j)K_j^{-1} = \left( \begin{array}{cc}
u & 0\\
0 &  u^{-1} \end{array}\right) \left( \begin{array}{cc}
a_1 & b_1\\
c_1 &   d_1\end{array}\right) \left( \begin{array}{cc}
u ^{-1}& 0\\
0 &  u \end{array}\right) =  \left( \begin{array}{cc}
a_1 & u^2b_1\\
u^{-2}c_1 &   d_1\end{array}\right) = \tilde\rho_2(\beta_j)
.$$
Therefore, $K_j$ conjugates the restriction of $\tilde\rho_1$ to $G_j$ to the restriction of
$\tilde\rho_2$ to $G_j$.
\end{proof}

We begin our analysis of trace normalized lifts which agree on $G_j$ by examining the relationship
between the images of elements of $\mathcal S$. Lemma \ref{generators} is the crucial step
in the proof of our main result. The additional information concerning matrix co-efficients in case (2) will
only be used when $S$ has genus two.

\begin{lemma}
\label{generators}
Suppose that $\rho_1^j:\pi_1(S)\to {\rm SL}(2,\mathbb C)$ and $\rho_2^j:\pi_1(S)\to {\rm SL}(2,\mathbb C)$
are trace normalized lifts, with respect to a standard generating set $\mathcal S$, of Kleinian
surface groups with the same complex marked length spectrum on $W(\alpha_1)$.
If $\rho^j_1$ and $\rho^j_2$ agree
on $G_j = <\alpha_j,\beta_j>$ and $\delta \in {\mathcal S}$, then  either
\begin{enumerate}
\item
$\rho^j_1(\delta)= \rho^j_2(\delta)$, or
\item
$\rho^j_1([\alpha_j,\beta_j]) $ is hyperbolic and
$R_j\rho^j_1(\delta)R_j^{-1}= \rho^j_2(\delta)$,
where $R_j$ is a lift of the rotation of order two about the axis of  $\rho^j_1([\alpha_j,\beta_j]) $.\newline
Moreover, if $\gamma \in G_j$ such that $\gamma, \gamma\delta\in W(\alpha_1)$, and
$$\rho^j_1([\alpha_j,\beta_j]) =  
 \left( \begin{array}{cc}
\lambda & 0\\
0& \lambda^{-1} \end{array}\right),\ \ 
\rho^j_1(\gamma) =  
 \left( \begin{array}{cc}
a & b\\
c & d  \end{array}\right), \ \ {\rm and}\ \ 
\rho_1^j(\delta) = 
 \left( \begin{array}{cc}
e & f\\
g & h \end{array}\right)
$$
where $|\lambda|>1$, then
$$\frac{e}{h} = -\frac{d}{a}$$
\end{enumerate}
\end{lemma}

\begin{proof}
The proof breaks up into two cases, depending on whether $\rho^j_1([\alpha_j,\beta_j])$ is hyperbolic or parabolic.

\medskip\noindent
{\bf Case I: $\rho^j_1([\alpha_j,\beta_j])$ is hyperbolic:} 
We may assume, by conjugating, that
$$\rho^j_1([\alpha_j,\beta_j]) = \rho^j_2([\alpha_j,\beta_j]) = \left( \begin{array}{cc}
\lambda & 0\\
0 &   \lambda^{-1}\end{array}\right) 
$$ 
where $|\lambda| > 1.$
As $\rho^j_1$ and $\rho^j_2$ agree on $G_j$, we can assume that $\delta$ is not either  $\alpha_j$ or $\beta_j$. Let 
$$\rho^j_i(\delta) = \left( \begin{array}{cc}
e_i & f_i\\
g_i &  h_i \end{array}\right).$$
where $e_ih_i-f_ig_i = 1.$
Moreover, since the traces  agree on generators, 
$$e_1 + h_1 = e_2+ h_2.$$

Suppose that  $\gamma \in G_j$ and $\gamma, \gamma\delta \in W(\alpha_1)$. 
(We can choose either $\gamma=\beta_j$ or $\gamma=\beta_j^{-1}$ depending on the  
orientations on the curves, see Figure \ref{dehntwist}). Let
$$\rho^j_1(\gamma) = \rho^j_2(\gamma) =  \left( \begin{array}{cc}
a & b\\
c &  d \end{array}\right).$$
where $ad-bc = 1$.

\begin{figure}[htbp] 
   \centering
   \includegraphics[width=5in]{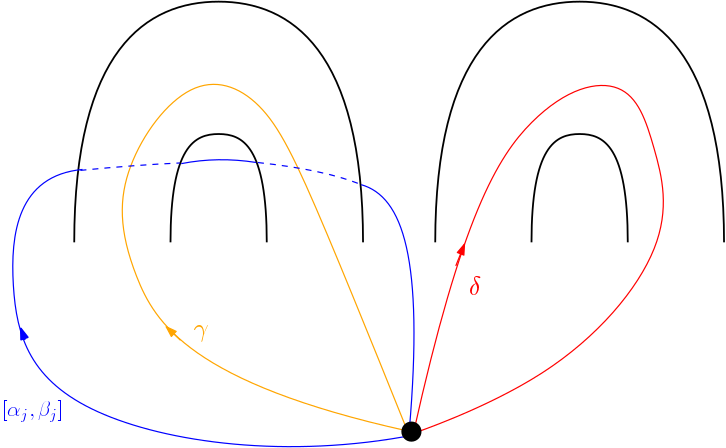} 
   \caption{Curves $\gamma, \delta$ and $[\alpha_j,\beta_j]$}
   \label{dehntwist}
\end{figure}

Let  $\gamma_n = [\alpha_j,\beta_j]^n \gamma [\alpha_j,\beta_j]^{-n} \delta,$
which also lies in $W(\alpha_1)$, 
since it is  the image of $\gamma\delta$ under the  $n$-fold Dehn twist
about $[\alpha_j,\beta_j]$.  
So, by assumption, 
$$\Tr(\rho^j_1(\gamma_n)) = \pm \Tr(\rho^j_2(\gamma_n)).$$

By expanding, we see that, for all $n$,
$$ae_1 + dh_1 + bg_1\lambda^{2n} + cf_1\lambda^{-2n} = \pm(ae_2 + dh_2 + bg_2\lambda^{2n} + cf_2\lambda^{-2n} ).$$
Thus there exists an increasing  subsequence where the traces either all agree or all differ up to sign.

If there  exists an increasing sequence $\{n_j\}$ such that the traces all agree, then
$$ae_1 + dh_1 + bg_1\lambda^{2n_j} + cf_1\lambda^{-2n_j} = ae_2 + dh_2 + bg_2\lambda^{2n_j} + cf_2\lambda^{-2n_j} $$
for all $n_j$.  Dividing the above equation by $\lambda^{2n_j}$ and taking a limit as $j\to \infty$, we
see that  $bg_1=bg_2$. Since $b\ne 0$, by Lemma \ref{basic comp}, $g_1=g_2$.
Thus,
$$ae_1 + dh_1 + cf_1\lambda^{-2n_j} = ae_2 + dh_2 + cf_2\lambda^{-2n_j} $$
and we see, by taking the limit as $n_j\to\infty$, that $ae_1+dh_1=ae_2+dh_2$.
Finally, since  $cf_1\lambda^{-2nj}=cf_2\lambda^{-2nj}$ and $\lambda$ and $c$ are non-zero, $f_1=f_2$.
Summarizing, we have 
\begin{equation}
ae_1 + dh_1 = ae_2+dh_2,\qquad f_1 = f_2,\qquad {\rm and}\quad g_1 = g_2.
\label{equaltraces}
\end{equation}

Similarly, if there  exists an increasing sequence $\{n_j\}$ such that the traces all disagree then,
$$
ae_1 + dh_1 + bg_1\lambda^{2n_j} + cf_1\lambda^{-2n_j} = -(ae_2 + dh_2 + bg_2\lambda^{2n_j} + cf_2\lambda^{-2n_j} )$$
for all $n_j$.  Taking limits as above, we conclude that
\begin{equation} 
ae_1+dh_1 = -(ae_2 + dh_2), \qquad f_1 = -f_2 , \qquad {\rm and}\quad  g_1 = -g_2.
\label{differtraces}
\end{equation}

Thus given any $\gamma \in G_j$ such that $\gamma, \gamma\delta \in W(\alpha_1)$, then $\gamma$ either satisfies equation (\ref{equaltraces}) or (\ref{differtraces}).
Since  $f_i$ and $ g_i$ are non-zero , we conclude that, with the above normalization, either
\begin{itemize}
\item[(a)] equation (\ref{equaltraces}) holds for all $\gamma  \in G_j$ such that $\gamma, \gamma\delta \in W(\alpha_1)$,
or
\item[(b)] equation (\ref{differtraces}) holds for all $\gamma  \in G_j$ such that $\gamma, \gamma\delta \in W(\alpha_1)$.
 \end{itemize}
 
 \medskip\noindent 
{\bf Case Ia): Equation (\ref{equaltraces}) holds for all $\gamma  \in G_j$ such that $\gamma, \gamma\delta \in W(\alpha_1)$}.
Choose $\gamma \in G_j$ such that $\gamma, \gamma\delta \in W(\alpha_1)$ and let 
$$\rho^j_1(\gamma) = \rho^j_2(\gamma) =  \left( \begin{array}{cc}
a & b\\
c &  d \end{array}\right).$$
where $ad-bc = 1$. Then, by equation (\ref{equaltraces}),
$$ae_1 + dh_1 = ae_2+dh_2,\qquad f_1 = f_2, \qquad  {\rm and}\quad g_1 = g_2.$$
Since  $e_ih_i-f_ig_i = 1$, we conclude that $e_1h_1 = e_2h_2$, so, since  we also have $e_1+h_1=e_2+h_2$,  
$$(x-e_1)(x-h_1)  = x^2 -2(e_1+h_1)x+e_1h_1= (x-e_2)(x-h_2),$$
which implies that either
(i) $e_1 = e_2$ and $h_1= h_2$ or (ii) $e_1 = h_2$ and $h_1 = e_2$.

If $e_1 = e_2$ and $h_1 = h_2$, then, since  we already know that $f_1=f_2$ and $g_1=g_2$, 
we may conclude that $\rho^j_1(\delta) = \rho^j_2(\delta)$ and we are in case (1).

If $e_1 = h_2$ and $h_1 = e_2$, then, since $ae_1+dh_1 = ae_2 + dh_2$, we conclude that
$$(a-d)(e_1-h_1) = 0.$$ 
If $a\ne d$,  then this implies that $e_1=h_1$, so $e_1=e_2$.
Since we already know that $f_1=f_2$ and $g_1=g_2$,  and all the matrix co-efficients are non-zero, we may
conclude that $\rho^j_1(\delta) = \rho^j_2(\delta)$, so we are again in case (1).

In order to conclude that we are in case (1), it only remains to check that we can  choose 
$\gamma_0\in G_j$, so that $\gamma_0,\gamma_0\delta\in W(\alpha_1)$ and
$$\rho^j_1(\gamma_0) = \rho^j_2(\gamma_0) =  \left( \begin{array}{cc}
a_0 & b_0\\
c_0 &  d_0 \end{array}\right).$$
where $a_0\ne d_0$.

We may assume that we have initially chosen $\gamma$ to be either $\beta_j^{-1}$ or $\beta_j$,
and that, with the above notation, $a=d$.  We may then choose $\gamma_0$ to be
the element in  $\{[\alpha_j,\beta_j]\gamma,[\alpha_j,\beta_j]^{-1}\gamma\}$ which is simple and has the property that 
$\gamma_0,\gamma_0\delta\in W(\alpha_1)$ (see Figure \ref{switch}). 
Observe that either $a_0=\lambda a$ and $d_0=\lambda^{-1}d$
or $a_0=\lambda^{-1}a$ and $d_0=\lambda d$, so $a_0\neq d_0$. This choice of $\gamma_0$ allows us to complete
the proof.
\begin{figure}[htbp] 
   \centering
   \includegraphics[width=5in]{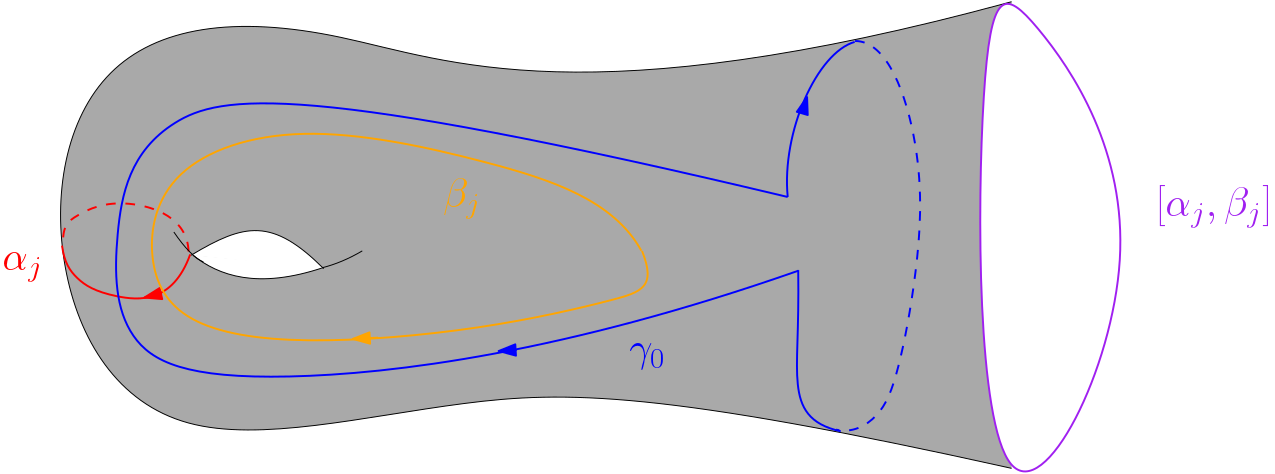} 
   \caption{The curve $\gamma_0$}
   \label{switch}
\end{figure}

\medskip\noindent
{\bf Case Ib): Equation (\ref{differtraces}) holds for all $\gamma  \in G_j$ such that $\gamma, \gamma\delta \in W(\alpha_1)$}. 
Choose $\gamma \in G_j$ so that $\gamma, \gamma\delta \in W(\alpha_1)$ and let 
$$\rho^j_1(\gamma) = \rho^j_2(\gamma) =  \left( \begin{array}{cc}
a & b\\
c &  d \end{array}\right).$$
where $ad-bc = 1$. Then, by equation (\ref{differtraces}).
$$ae_1 + dh_1 = -(ae_2+dh_2),\qquad f_1 = -f_2,\qquad  {\rm and}\quad g_1 = -g_2.$$
As  $e_ih_i-f_ig_i = 1$,  we conclude that $e_1h_1 = e_2h_2$.
Since 
$e_1h_1 = e_2h_2$ and $e_1 + h_1 = e_2 + h_2$, we may conclude, just as in Case Ia, that
either (i) $e_1 = h_2$ and $h_1 = e_2$ or  (ii) $e_1 = e_2,$ and $h_1= h_2$.

If $e_1 = h_2$ and $h_1 = e_2$, then, since $ae_1+dh_1 = -(ae_2 + dh_2)$, we see that
$$\Tr(\rho^j_1(\gamma))\Tr(\rho^j_1(\delta)) =  (a+d)(e_1+h_1)=ae_1+dh_1+ae_2+dh_2=0$$
which is impossible since both traces are non-zero. Therefore, $e_1 = e_2$ and $h_1 = h_2$, so.
since $f_1=-f_2$ and $g_1=-g_2$, 
$$ \rho^j_2(\delta) = \left( \begin{array}{cc}
e_1 & -f_1\\
-g_1 & h_1  \end{array}\right)  = 
\left( \begin{array}{cc}
i & 0\\
0 & -i  \end{array}\right) 
 \left( \begin{array}{cc}
e_1 & f_1\\
g_1 & h_1  \end{array}\right)
\left( \begin{array}{cc}
-i & 0\\
0 & i  \end{array}\right) = R_j\rho^j_1(\delta) R_j^{-1}
$$
so
$ \rho^j_1(\delta)=R_j\rho^j_2(\delta)R_j^{-1}$ and 
$R_j$ is a lift of the rotation of order two about the axis of 
$\rho^j_1([\alpha_j,\beta_j])$. 

Moreover, since $ae_1+dh_1 = -(ae_2 + dh_2)=-(ae_1+dh_1)$, we see that
\hbox{$ae_1+dh_1=0$}, so
$$\frac{e_1}{h_1} = -\frac{d}{a},$$
and we are in case (2).

\medskip\noindent
{\bf Case II: $\rho^j_1([\alpha_j,\beta_j])$ is  parabolic:}  
Choose $\gamma \in G_j$ so that $\gamma, \gamma\delta \in W(\alpha_1)$. 
We may conjugate so that
$$\rho^j_1([\alpha_j,\beta_j]) = \rho^j_2([\alpha_j,\beta_j]) = \left( \begin{array}{cc}
1 &1\\
0 &   1\end{array}\right) 
$$ 
and $0$ is a  fixed points of $\rho^j_1(\gamma) = \rho^j_2(\gamma)$, so
$$\rho^j_1(\gamma) = \rho^j_2(\gamma) =  \left( \begin{array}{cc}
a & 0\\
c &  d \end{array}\right)$$
where $ad = 1$.
So, $a$ and $d$ are non-zero and $c$ is non-zero,
since otherwise  $\rho^i_1([\alpha_i,\beta_i])$ and $\rho^i_1(\gamma)$ would have a common fixed point.

Let 
$$\rho^j_i(\delta) = \left( \begin{array}{cc}
e_i & f_i\\
g_i &  h_i \end{array}\right).$$
where $e_ih_i-f_ig_i = 1$ and all the matrix coefficients are non-zero.
Moreover, since the lifts are trace normalized with respect to $\mathcal S$,
$$e_1 + h_1 = e_2+ h_2.$$

Again, let  $\gamma_n = [\alpha_i,\beta_i]^n \gamma [\alpha_i,\beta_i]^{-n} \delta,$
which  lies in $W(\alpha_1)$, so
$$\Tr(\rho^i_1(\gamma_n)) = \pm {\rm Tr}(\rho^i_2(\gamma_n)).$$
Expanding, we obtain
$$(a+nc)e_1 + (n(d-a)-n^2c)g_1 + cf_i +(d-nc)h_1 = \pm((a+nc)e_2 + (n(d-a)-n^2c)g_2 + cf_2 +(d-nc)h_2).$$

If there are infinitely many  $n$ where the traces differ, then, by dividing by $n^2$ and taking limits, 
we conclude that $cg_1 = -cg_2$, so, since $c\ne 0$, $g_1 = - g_2$.  By, successively taking limits, we further
conclude that
$$c(e_1-h_1) = -c(e_2-h_2)\qquad {\rm and}\qquad ae_1+cf_1+dh_1 = -ae_2-cf_2-dh_2$$
Since $c\ne 0$, we see that $e_1 - h_1 = -e_2+h_2$. As $e_1+ h_1 = e_2 + h_2$, this implies that
$e_1 = h_2$ and $h_1 = e_2$. Therefore, 
$$f_1g_1 = e_1h_1-1 = e_2h_2-1 = f_2g_2 = -f_2g_1. $$
As $\rho^j_i(\delta)$ does not have a fixed point at $\infty$, $g_1$ must be non-zero, so $f_1 = -f_2$.
Since $ae_1+cf_1+dh_1 = -ae_2-cf_2-dh_2$, this implies that
$${\rm Tr}(\gamma){\rm Tr}(\delta) = (a+d)(e_1+h_1) = 0$$
which is a contradiction, since both traces are non-zero. 

Therefore, there are infinitely many  $n$ where the traces agree,  so, taking limits  as above,
we see that 
$$g_1 = g_2,\qquad  c(e_1-h_1) = c(e_2-h_2),\qquad {\rm and}\qquad ae_1+cf_1+dh_1 = ae_2+cf_2+dh_2.$$
Since $c\ne 0$, $e_1 - h_1 = e_2-h_2$, which we may combine with the fact that  $e_1+ h_1 = e_2 + h_2$,
to conclude that $e_1 = e_2$ and $ h_1 = h_2$. Therefore,  
$$f_1g_1 = e_1h_1-1 = e_2h_2-1 = f_2g_2 = f_2g_1$$
which implies that  $f_1 = f_2$, so $\rho^j_1(\delta) =  \rho^j_2(\delta)$, which implies that we are in case (1).
\end{proof}

We now refine our analysis of trace normalized lifts which agree on $G_j$ to show that, for all $k$,
they either agree on $G_k$ or differ by a rotation in the axis of  the image of the commutator in $G_j$.

\begin{lemma}
\label{step one}
Suppose that $\rho_1:\pi_1(S)\to {\rm SL}(2,\mathbb C)$ and $\rho_2:\pi_1(S)\to {\rm SL}(2,\mathbb C)$
are trace normalized lifts, with respect to a standard generating set $\mathcal S$, of Kleinian
surface groups with the same marked complex length spectrum on $W(\alpha_1)$.
If $j,k\in\{1,\ldots, g\}$ and $\rho^j_1$ and $\rho^j_2$ agree
on $G_j = <\alpha_j,\beta_j>$, then either
\begin{enumerate}
\item
$\rho_1^j$ and $\rho_2^j$ agree on $G_k$, or
\item
$\rho_1^j([\alpha_j,\beta_j])$ is hyperbolic and 
$\rho_1^j|_{G_k}=(R_j\rho_2^j R_j^{-1})|_{G_k}$ where $R_j$ is a lift of the rotation of order two about the axis
of $\rho_1^j([\alpha_j,\beta_j])$.
\end{enumerate}
\end{lemma}

\begin{proof}
If neither (1) or (2) holds, then
Lemma \ref{generators}  implies that, perhaps after switching $\alpha_k$ and $\beta_k$,
$\rho_1^j([\alpha_j,\beta_j])$ is hyperbolic,
$$\rho_2^j(\alpha_k) = \rho_1^j(\alpha_k) \qquad{\rm and}\qquad \rho_2^j(\beta_k)=R_j\rho_1(\beta_k)R_j^{-1}$$
where $R_j$ is a lift of the rotation of order two about the axis
of $\rho_1^j([\alpha_j,\beta_j])$.

Lemma \ref{equal on Gi} implies that  there exists $K\in {\rm SL}(2,\mathbb C)$ so that 
$\rho^j_1$ and $K\rho^j_2K^{-1}$ agree on $G_k$. If $K=\pm I$, then $\rho_1^j$ and $\rho_2^j=K\rho_2^jK^{-1}$
agree on $G_k$ and we are in case (1). So, we may assume that $K\ne\pm I$ and
$$\rho^j_1(\alpha_k)=K\rho^j_2(\alpha_k)K^{-1}=K\rho^j_1(\alpha_k)K^{-1}\quad {\rm and}\quad
\rho_1^j(\beta_k)=K\rho_2^j(\beta_k)K^{-1}=KR_j\rho_1^j(\beta_k)R_j^{-1}K^{-1}.$$
It follows that $K$ fixes the fixed points of $\rho^j_1(\alpha_k)$ and  that $KR_j$ fixes the fixed points	
of $\rho^j_1(\beta_k)$. 

If $\rho^j_1$ and $K\rho_2^jK^{-1}$ agree on either $ \alpha_j$ or $\beta_j$,
then, since $\rho_1^j$ and  $\rho_2^j$ agree on $G_j$, either
$$\rho^j_1(\alpha_j)=K\rho_2^j(\alpha_j)K^{-1}=K\rho_1^j(\alpha_j)K^{-1}\ \ {\rm or}\ \
\rho^j_1(\beta_j)=K\rho_2^j(\beta_j)K^{-1}=K\rho_1^j(\beta_j)K^{-1}$$
which implies that $K$ fixes the fixed points of either $\rho^j_1(\alpha_j)$ or $\rho_1^j(\beta_j)$.
But, since $K$ fixes the fixed points of $\rho^j_1(\alpha_k)$,  and the fixed points of
$\rho_1^j(\alpha_k)$, $\rho^j_1(\alpha_j)$ and $\rho_1^j(\beta_j)$ are all distinct,
this implies that $K=\pm I$,
which is a contradiction.

Therefore, $\rho^j_1$ and $K\rho_2^jK^{-1}$ disagree on both $ \alpha_j$ and $\beta_j$.
Lemma \ref{generators} then implies that  $\rho_1^j([\alpha_k,\beta_k])$ is hyperbolic
and 
$$\rho^j_1(\alpha_j)=R_kK\rho_2(\alpha_j)K^{-1}R_k^{-1}\quad{\rm and}\quad
\rho_1^j(\beta_j)=R_kK\rho_2^j(\beta_j)K^{-1}R_k^{-1}$$
where $R_k$ is a lift of the rotation of order two about the axis of $\rho_1^j([\alpha_k,\beta_k])$.
Therefore, $\rho^j_1$ and $R_kK\rho_2^jK^{-1}R_k^{-1}$ agree on $G_j$.
Since $\rho_1^j$ and $\rho_2^j$ agree on $G_j$ this implies that  $R_kK=\pm I$, so $K=\pm R_k^{-1}$. 
Therefore, $K$ fixes the fixed points of $\rho_1^j([\alpha_k,\beta_k])$.
However, since we already know that $K$ fixes the fixed points of $\rho^j_1(\alpha_k)$, this implies that $K=\pm I$,
which is again a contradiction.
\end{proof}

It is now relatively simple to use Lemma \ref{step one} to complete the proof when $S$ has genus
at least three.

\begin{lemma}
\label{general case}
Suppose that $S$ has genus $g\ge 3$  and  that $\tilde\rho_1:\pi_1(S)\to {\rm SL}(2,\mathbb C)$ and 
\hbox{$\tilde\rho_2:\pi_1(S)\to {\rm SL}(2,\mathbb C)$}
are trace normalized lifts, with respect to a standard generating set $\mathcal S$, of Kleinian
surface groups with the same marked complex length spectrum on $W(\alpha_1)$.
Then,
$\tilde\rho_1$ is conjugate to $\tilde\rho_2$ in ${\rm SL}(2,\mathbb C)$.
\end{lemma}

\begin{proof} 
Lemma \ref{equal on Gi} implies that we may choose conjugates
$\rho^1_1$ and $\rho^1_2$ of  $\tilde \rho_1$ and $\tilde\rho_2$ which agree
on $G_1 = <\alpha_j,\beta_j>$ .
Lemma \ref{step one} implies that for all $j>1$, the restrictions of $\rho_1^1$ and $\rho_2^1$ to $G_j$
either agree or are conjugate by $R_1$, where $R_1$ is a lift of the rotation of order two
about the axis of $\rho_1^1([\alpha_1,\beta_1])$. If $\rho_1^1\ne\rho_2^1$, we may assume without
loss of generality that  $\rho^1_1|_{G_2} = (R_1\rho^1_2 R_1^{-1})|_{G_2}$.

Lemma \ref{step one} implies that either (1)
$\rho_1^1$ and $ R_1\rho_2^1R_1^{-1}$ agree on $G_1$,  or (2) $\rho_2^1([\alpha_2,\beta_2])$ is hyperbolic and
$\rho_1^1$ and $ R_1\rho_2^1R_1^{-1}$ are
conjugate by $R_2$ on $G_1$ where $R_2$ is a lift of the rotation of order two about the axis of $\rho^1_1([\alpha_2,\beta_2])$.

If $\rho_1^1$ and $R_1\rho_2^1R_1^{-1}$  agree on $G_1$, then, since  $\rho^1_1$ 
and $\rho^1_2$ agree on $G_1$, $R_1$ commutes with every element of $\rho_1^1(G_1)$.
Since  $\rho_1^1(G_1)$ is non-elementary, this implies that $R_1=\pm I$, which is a contradiction.
 
If $\rho^1_1$ and $R_2(R_1\rho_2^1R_1^{-1})R_2^{-1}$  agree on $G_1$, then
since $\rho_1^1$ and $\rho^1_2$ agree on $G_1$, we similarly
conclude that  $R_2R_1 = \pm I$. So,  $R_2 =\pm R_1^{-1}$. This would imply that
$\rho^1_1([\alpha_1,\beta_1])$ and $ \rho^1_1([\alpha_2,\beta_2])$ have the same axis, 
so share fixed points,
which is a contradiction unless $S$ has genus 2 and $[\alpha_1,\beta_1]=[\alpha_2,\beta_2]^{-1}$.
\end{proof}

We now complete the proof by establishing our result in the genus two case.

\begin{lemma}
\label{genus two case}
Suppose that $S$ has genus $g=2$ and that $\tilde\rho_1:\pi_1(S)\to {\rm SL}(2,\mathbb C)$ and 
\hbox{$\tilde\rho_2:\pi_1(S)\to {\rm SL}(2,\mathbb C)$}
are trace normalized lifts, with respect to a standard generating set $\mathcal S$, of Kleinian
surface groups with the same marked complex length spectrum on $W(\alpha_1)$.
Then, $\tilde{\rho}_1$ is conjugate to $\tilde{\rho}_2$ in ${\rm SL}(2,\mathbb C)$.
\end{lemma}

\begin{proof}
In this case $\pi_1(S) = < \alpha_1, \beta_1, \alpha_2, \beta_2>$ given by the standard oriented curves as in Figure \ref{g2}.
Lemma \ref{equal on Gi} implies that we may conjugate  $\tilde{\rho}_i$ to $\rho^1_i$ so
that $\rho_1^1$ and  $\rho_2^1$ agree on $G_1$. 

Suppose that $\rho_1^1\ne\rho^1_2$.
Lemma \ref{step one} implies that $\rho_1([\alpha_1,\beta_1])$ is hyperbolic and
$\rho_1^1$ and $R_1\rho_2^1R_1^{-1}$ agree on $G_2$ where $R_1$ is a lift of the rotation of order two about the
axis of  $\rho_1([\alpha_1,\beta_1])$. We may normalize so that
$$\rho^1_1([\alpha_1,\beta_1]) = \rho^1_2([\alpha_1,\beta_1]) = \left( \begin{array}{cc}
\lambda & 0\\
0 &   \lambda^{-1}\end{array}\right) 
$$ 
where $|\lambda| > 1.$

\begin{figure}[htbp] 
   \centering
   \includegraphics[width=5in]{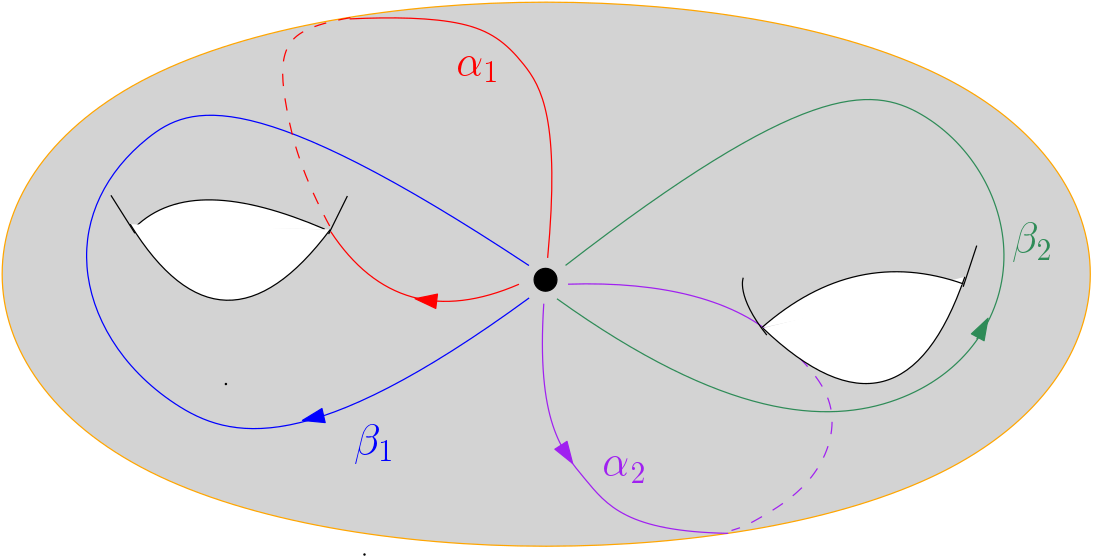} 
   \caption{Surface of genus 2 with standard generators}
   \label{g2}
\end{figure}

Let $\gamma = \alpha_1^{-1}$.  Notice that, assuming that elements of $\mathcal S$ are oriented as
in Figure \ref{g2},  $\gamma \alpha_2$ and $\gamma\beta_2$  are 
simple and lie in $W(\alpha_1)$.
Define $A = \rho^1_1(\alpha_1^{-1})$,  $C = \rho^1_1(\alpha_2)$,  and $ D = \rho^1_1(\beta_2)$, and
adapt the convention that if a matrix is denoted by capital letter $X$, then it has coefficients $(x_{ij})$.
Since $\rho_1^1$ and $\rho_2^1$ differ on both $\alpha_2$ and $\beta_2$,
Lemma  \ref{generators} implies that 
$$\frac{c_{11}}{c_{22}} = \frac{d_{11}}{d_{22}} =  -\frac{a_{22}}{a_{11}}.$$

We now consider the element $ \alpha_2\beta_2 \in G_2$. We replace generators $\{\alpha_2, \beta_2\}$ by 
$\{\alpha_2\beta_2, \alpha_2^{-1}\}$ in $\mathcal S$ to form another standard generating set 
${\mathcal S}' = \{\alpha_1,\beta_1,  \alpha_2\beta_2, \alpha_2^{-1}\}$ 
(we note that $[\alpha_2\beta_2,\alpha_2^{-1}] = [\alpha_2,\beta_2]$). The representations $\rho^1_1$ and $\rho^1_2$
agree  on $G_1$ and their traces agree on ${\mathcal S}'$.
Notice that  $\gamma(\alpha_2\beta_2) \in W(\alpha_1)$ (see Figure \ref{g2}) and
let $E= CD =\rho_1(\alpha_2\beta_2)$. Since $\rho_1^1$ and $\rho_2^1$ agree on $G_1$ and disagree on
$\alpha_2\beta_2$,
Lemma  \ref{generators} implies that
$$\frac{e_{11}}{e_{22}} = -\frac{a_{22}}{a_{11}}.$$
If $M =CDC^{-1}D^{-1}=\rho_1^1([\alpha_2,\beta_2]^{-1})=\rho_1^1([\alpha_1,\beta_1])$, then 
$$M = \left( \begin{array}{cc}
\lambda & 0\\
0 & \lambda^{-1} \end{array}\right) 
$$
Since $DC = M^{-1}CD = M^{-1}E$ and $C^{-1}(CD)C=DC$,
$$\Tr(DC) = \Tr(CD) = \Tr(E) = \Tr(M^{-1}E),$$
so $e_{11}+e_{22} =  \lambda^{-1} e_{11} + \lambda e_{22}$. Therefore
$$\frac{c_{11}}{c_{22}} = \frac{d_{11}}{d_{22}} =\frac{e_{11}}{e_{22}} = \frac{(\lambda-1)}{(1-\lambda^{-1})} = \lambda.$$
Since $CD = MDC$,
$$ (CD)_{12} = c_{11}d_{12} + c_{12}d_{22} = (MDC)_{12} = \lambda(DC)_{12} = \lambda(d_{11}c_{12}+d_{12}c_{22}),$$
so, since $c_{11} = \lambda c_{22}$ and $d_{11} = \lambda d_{22}$, 
$$\lambda c_{22}d_{12}+ c_{12}d_{22} = \lambda(\lambda d_{22}c_{12}+ d_{12}c_{22}),$$
which implies that
$$(\lambda^2 -1)c_{12}d_{22} = 0.$$
Thus, since $d_{22},$ and $c_{12}$ are both non-zero, $\lambda^{2} = 1$ which is a contradiction. 
Therefore, it must be the case that $\rho_1^1=\rho_2^1$.
\end{proof}

\medskip\noindent
{\bf Remark:}  Observe that in the proof we only assume that  there exists a simple non-separating
curve $\alpha$ such that $\rho_1$ and $\rho_2$ have the same or conjugate marked complex
length spectrum on $W(\alpha)$. It then follows from the remark at the end of Section \ref{complex length}
that $\rho_1$ and $\rho_2$ are conjugate in the isometry group of $\mathbb H^3$
if  they have the same marked length spectrum on all simple, non-separating curves.

\section{Acylindrical 3-manifolds}
\label{acylindrical}

We now use our main result to prove our length rigidity theorem
for acylindrical \hbox{3-manifolds}.

\medskip\noindent
{\bf Theorem \ref{acyl app}.}
{\em If $M$ is a compact, acylindrical, hyperbolizable  3-manifold, and \hbox{$\rho_1:\pi_1(M)\to {\rm PSL}(2,\mathbb C)$}
and $\rho_2:\pi_1(M)\to {\rm PSL}(2,\mathbb C)$ are two discrete faithful representations, such
that $\ell_{\rho_1}(\alpha)=\ell_{\rho_2}(\alpha)$ if $\alpha\in\pi_1(M)$ is represented by a simple
closed curve on $\partial M$, then $\rho_1$ is conjugate to either $\rho_2$ or $\bar\rho_2$.}

\medskip
In order to motivate the more general proof, we offer a sketch of the proof in the case 
that $M$ has connected boundary and $\rho_1$ and $\rho_2$
are both convex cocompact, i.e. each \hbox{$N_i=\mathbb H^3/\rho_i(\pi_1(M))$} admits a conformal
compactification with boundary $\partial_cN_i$.
After perhaps replacing $\rho_2$ with a complex conjugate representation, we may assume
by Johannson's Classification Theorem  \cite{johannson},
that there exists an orientation-preserving
homeomorphism $j:N_1\to N_2$ in the homotopy class determined by $\rho_2\circ\rho_1^{-1}$.
In this case, work of Bers \cite{bers-survey} implies that $\rho_1$ is conjugate to $\rho_2$ if and only if
$j$ extends, up to isotopy, to a conformal homeomorphism $\hat j:\partial_cN_1\to\partial_cN_2$.

If the length spectra of $\rho_1$ and $\rho_2$ agree on simple  curves on the boundary, our main result
implies that the restrictions  of $\rho_1$  and $\rho_2$ to the fundamental group of the boundary agree up to conjugacy in
in the isometry group of $\mathbb H^3$. It follows that the conformal compactifications $N^S_1$ and $N^S_2$ of the covers of
$N_1$ and $N_2$ associated to $\pi_1(S)$ are either conformal or anti-conformally homeomorphic
(by a homeomorphism in the homotopy class consistent with the identifications of their fundamental groups
with $\pi_1(S)$.)
Notice that $\partial_cN_i$ is identified with one component of the conformal boundary of $N_i^S$.
If the conformal compactifications of $N_1^S$ and $N_2^S$ are conformally homeomorphic, then $j$
extends to a conformal homeomorphism between $\partial_cN_1$ and $\partial_cN_2$, so 
Bers' result implies that
$\rho_1$ and $\rho_2$ are conjugate in ${\rm PSL}(2,\mathbb C)$. If not,
then one may use the Klein-Maskit combination Theorems
(see Maskit \cite{maskit}) to combine $\rho_1$ and  (a conjugate of) $\rho_2$ to produce the holonomy 
representation of a hyperbolic structure
on the double of $M$. Mostow's Rigidity Theorem \cite{mostow} then implies that the involution of the double is
an isometry, so  the restriction of $\rho_1$ and $\rho_2$ to the boundary is Fuchsian, and the conformal
structures on the boundary must agree. 

In the general case, we must organize the components of the boundary
into those where the conformal structures agree and those where they disagree and replace the use
of Bers' theorem and Mostow's Rigidity Theorem with applications of the Ending Lamination Theorem \cite{ELC1,ELC2}.

\medskip\noindent
{\em Proof of Theorem \ref{acyl app}:}
Suppose that $\rho_1:\pi_1(M)\to {\rm PSL}(2,\mathbb C)$ and
$\rho_2:\pi_1(M)\to {\rm PSL}(2,\mathbb C)$ are two discrete faithful representations, such
that $\ell_{\rho_1}(\alpha)=\ell_{\rho_2}(\alpha)$ if $\alpha\in\pi_1(M)$ is represented by a simple
closed curve on $\partial M$.  Let $N_i=\mathbb H^3/\rho_i(\pi_1(M))$. 

There exists a homotopy
equivalence $h_i:M\to N_i$ in the homotopy class determined by $\rho_i$ (where we regard $\rho_i$
as an isomorphism between $\pi_1(M)$ and $\pi_1(N_i)=\rho_i(\pi_1(M))$).
Bonahon's Tameness Theorem
\cite{bonahon-ends} implies that $N_i$ may be identified with the interior of a compact 3-manifold $M_i$.
Since $M$ is acylindrical, Johannson's Classification Theorem
\cite{johannson} implies that $h_i$ is homotopic to a homeomorphism.
Therefore, there exists a homeomorphism $g:M_1\to M_2$ so that $g\circ h_1$ is homotopic to $h_2$, i.e.
$g$ is in the homotopy class determined by $\rho_2\circ\rho_1^{-1}$. If $g$ is not orientation-preserving
then we replace $\rho_2$ with $\bar\rho_2$.  So, we may assume that $g$ is orientation-preserving.

Thurston associates to each $\rho_i$ a family of ending invariants, i.e. a multicurve $p_i$, called the parabolic locus,
in the collection $\partial_0M_i$  of non-toroidal components of $\partial M_i$,
and on each component of $\partial_0M_i-p_i$, either a  hyperbolic structure of finite area (in which
case the component is called geometrically finite) or a filling, geodesic
lamination which admits a transverse measure of full support, called the ending lamination (in which case the component
is called geometrically infinite),
see \cite{ELC1,ELC2} for details. The Ending Lamination Theorem of Minsky \cite{ELC1} and Brock-Canary-Minsky \cite{ELC2}
implies that $\rho_1$ is conjugate to $\rho_2$ in ${\rm PSL}(2,\mathbb C)$ if and only if
the ending invariants of $\rho_1$ agree with the ending invariants of $\rho_2$, i.e., up to isotopy of $g$, $g(p_1)=p_2$, the restriction of $g$ to each geometrically finite component of $\partial_0M_1-p_1$ 
is an isometry onto a geometrically finite component of $\partial_0 M_2-p_2$, and
if a geometrically infinite component of $\partial_0 M_1-p_1$ has ending lamination $\lambda$, then $g(\lambda)$
is the ending lamination of a geometrically infinite component of $\partial_0 M_2-p_2$.

Let $\{S_1,\ldots,S_r\}$ be the components of $\partial_0 M_1$.
Since $\ell_{\rho_1}(\alpha)=\ell_{\rho_2}(\alpha)$ if $\alpha\in\pi_1(M)$ is represented by a simple
closed curve on $\partial M$, our main result, Theorem \ref{main}, implies that, for each $j$, 
$\rho_1|_{\pi_1(S_j)}$ is conjugate  to either $\rho_2$ or $\bar\rho_2$.
If each $\rho_1|_{\pi_1(S_j)}$ is conjugate  to $\rho_2|_{\pi_1(S_j)}$,
then the ending invariants of $\rho_1$ agree with the ending invariants of $\rho_2$, so
$\rho_1$ is conjugate to $\rho_2$ in ${\rm PSL}(2,\mathbb C)$ and we are done.

If there exists $S_j$ so that $\rho_1|_{\pi_1(S_j)}$ is conjugate to $\bar\rho_2|_{\pi_i(S_j)}$, 
then we may re-order the components of $\partial M$,
so that  $\rho_1|_{\pi_1(S_j)}$ is conjugate to $\rho_2|_{\pi_i(S_j)}$  if and only if
$j>q$ for some $q\in \{1,\ldots, r-1\}$. Let $N_i^j$ be the cover of $N_i$ associated to $\pi_1(S_j)$.
Then, $N_i^j$ may be identified with the interior of $S_j\times [0,1]$ so that the end invariants of $N_i^j$ 
on $S_j\times \{ 1\}$ agree with the restriction of the end invariants of $N_i$ restricted to $S_j$.
The Covering Theorem (see \cite{canary-cover}) may be used to show that the restriction of the ending
invariant  of $N_i^j$ to $S_j\times \{0\}$ is a conformal structure on all of $S_j\times \{0\}$
(see the discussion in Kent's proof of Thurston's Bounded Image Theorem \cite[Thm. 41]{kent-skinning}).
If $j\le q$, then the ending invariant of $N_1^j$ on $S_j\times \{1\}$ agree with the end invariants of $N_2^j$
on $S_j\times \{0\}$, while the ending invariant of $N_1^j$ on $S_j\times \{0\}$ agree with the end invariants of $N_2^j$
on $S_j\times \{1\}$. It follows that, if $j\le q$, then the ending invariants of $N_i^j$ are a pair of hyperbolic
structures on $S_j\times \{0\}$ and $S_j\times \{1\}$, so $\rho_i|_{\pi_1(S_j)}$ is quasifuchsian, and the map
from $S_j\times \{0,1\}$ to itself given by taking $(x,t)$ to $(x,1-t)$ is isotopic to an isometry from
the ending invariants of $N_1^j$ to the ending invariants of $N_2^j$.

Let $\hat M$ be the manifold obtained by doubling $M$ along $S_1\cup\cdots \cup S_{q-1}$. We first conjugate
$\bar\rho_2$  so that  $\rho_1(\pi_1(S_1))=\bar\rho_2(\pi_1(S_1))$.
We may use the Klein-Maskit
Combination Theorems \cite{maskit} to see that the combination of   $\rho_1(\pi_1(M))$ and $\bar\rho_2(\pi_1(M))$
gives rise to a hyperbolic structure on the double of $M$ along $S_1$. 
For each $j=2,\ldots,q$, let $A_j\in{\rm PSL}(2,\mathbb C)$
be chosen so that $A_j$ conjugates $\rho_1|_{\pi_1(S_j)}$ to $\rho_2|_{\pi_1(S_j)}$. Then, the Klein-Maskit
Combination Theorems give that the group generated by $\rho_1(\pi_1(M))$, $\rho_2(\pi_1(M))$ and $\{A_2,\ldots A_q\}$
gives rise to a hyperbolic structure on $\hat M$ and so a representation $\hat\rho:\pi_1(\hat M)\to {\rm PSL}(2,\mathbb C)$.
(See Sections 8 and 9 of Morgan \cite{morgan}, in particular Theorem 9.4, for a discussion of
the Klein-Maskit Combination Theorems in a topological phrasing compatible with our application.) 
The obvious involution of $\hat M$ preserves the ending invariants of $\hat N=\mathbb H^3/\hat\rho(\pi_1(\hat M))$,
so, by the Ending Lamination Theorem, there is an isometry of $\hat N$ realizing this involution. By restriction,
$\rho_1$ is conjugate to $\rho_2$.
\qed

\medskip\noindent
{\bf Remark:} In the case that $\partial_0 M$ is connected, Ian Agol astutely pointed out that one may rephrase
this proof as an application of the fact that the square of the skinning map has a unique fixed point,
see Morgan \cite{morgan} and Kent \cite{kent-skinning} for a discussion of the skinning map.

\section{Renormalized Pressure Intersection}
\label{pressure} 

We next show that the isometry group of the renormalized pressure intersection is generated by the
(extended) mapping
class group and complex conjugation. We begin by reinterpreting our renormalized pressure intersection
in terms of the Patterson-Sullivan geodesic  current, following Bridgeman \cite{Bri10} and 
Bridgeman-Taylor \cite{BT08}. Since isometries of the renormalized pressure intersection  are isometries
of the pressure metric and  the only degenerate vectors for the pressure metric are at points
on the Fuchsian locus (see Bridgeman \cite{Bri10}), any isometry $f$ of the renormalized pressure intersection 
must preserve the Fuchsian locus. Since the restriction of the pressure metric to the Fuchsian locus is a multiple of the
Weil-Petersson metric,
we may apply Masur and Wolf's classification of isometries of the Weil-Petersson metric \cite{masur-wolf}, to conclude
that the restriction of $f$ to the Fuchsian locus agrees with the action of a mapping class $g$.
We  then use
Bonahon's interpretation (see \cite{Bon88}) of the Thurston compactification of Teichm\"uller space 
in terms of  geodesic currents, to show that if $\rho\in QF(S)$, then $f(\rho)$ and $g(\rho)$  have proportional
simple marked  length spectrum. The proof is completed by showing that any two Kleinian surface groups
with proportional simple marked length spectrum are conjugate in the isometry group of $\mathbb H^3$.

\subsection{Geodesic Currents}
Identify the space  $G(\Hp)$  of  geodesics in $\Hp$ with the open M\"obius band
$( \Sph^{1}_{\infty} \times  \Sph^{1}_{\infty} -\mbox{ diagonal})/\mathbb Z_2$ by identifying a geodesic 
with its endpoints. A {\em geodesic current} on
a hyperbolic surface $X = \Hp/\Gamma$ is a positive Borel measure on $G( \Hp)$
that is invariant under the action of $\Gamma$. For example, if $\alpha$ is a primitive closed geodesic
on $X$, we may consider the geodesic current $C_\alpha$ given by taking the Dirac measure on
the pre-image of $\alpha$ in $G(\Hp)$.
Let ${\cal C}(X)$ be the space of geodesic currents defined  for $X = \Hp/\C$ with the weak*-topology.
If $S$ is a closed surface of genus at least two, we identify $S$ with a fixed hyperbolic surface $X_0$
and let ${\cal C}(S)={\cal C}(X_0)$.

Following Bonahon \cite{Bon88}, Bridgeman and Taylor \cite[Sec. 5]{BT08} showed that there is a natural continuous, 
linear  function
$$L:{\mathcal C}(S)\times QF(S)\to \mathbb R,$$
called the {\em length function},
with the property that  if $rC_\alpha$ is a positive multiple of
$C_\alpha$, then 
$L(rC_\alpha,\rho)=r\ell_\rho(\alpha).$
Since multiples of closed geodesics are dense
in $\mathcal C(S)$ (see \cite[Prop. 4.2]{bonahon-ends}), this property completely determines $L$.

Given $\rho\in QF(S)$, Sullivan \cite[Prop. 11, Thm. 21]{Sul79}, following work of Patterson \cite{Pat76} in the Fuchsian case,
used the Poincar\'e series to define a  non-atomic $\Gamma$-invariant
measure $\mu_\rho$ on $\partial_\infty\mathbb H^3\times\partial_\infty\mathbb H^3$, called the {\em Patterson-Sullivan measure},
which is supported on $\Lambda(\rho)\times\Lambda(\rho)$, where $\Lambda(\rho)$ is the limit set of $\rho(\pi_1(S))$,
such that $\Gamma$ acts ergodically on $\partial_\infty\mathbb H^3\times\partial_\infty\mathbb H^3$ with respect to $\mu_\rho$
(see also Sullivan \cite[Thm. 3]{sullivan-entropy}). One may push forward $\mu_\rho$ to obtain a $\Gamma$-invariant
measure $\hat\mu_\rho$ on $\Lambda(\rho)\times\Lambda(\rho)/\mathbb Z_2$ (where $\mathbb Z_2$ acts by 
interchanging the factors).
The representation $\rho$ induces a homeomorphism $f_\rho:S^1\to \Lambda(\rho)$, where we use our identification
of $S$ with $X_0$ to identify $S^1$ with $\partial_\infty \mathbb H^2$, so one obtains a homeomorphism
$\widehat{f_\rho\times f_\rho}:\mathcal G(\mathbb H^2)\to \Lambda(\rho)\times\Lambda(\rho)/\mathbb Z_2$.
One then defines
the {\em Patterson-Sullivan current} as the scalar multiple $\psi(\rho)\in\mathcal C(S)$ of
 $\widehat{(f_\rho\times f_\rho)}_*(\hat\mu_\rho)$  so that $L(\psi(\rho),\rho)=1$.
If $\rho$ is Fuchsian, then $\psi(\rho)$ is the Liouville geodesic current constructed by Bonahon \cite{Bon88}
(see \cite[Thm. 1]{Ham02}).
Hamenstadt \cite[Thm. 1]{Ham02} showed that the associated map 
$$\psi:QF(S)\to \mathcal{C}(S)$$
is continuous and injective. 
See Bridgeman \cite{Bri10}, Bridgeman-Taylor \cite{BT08} or Hamenstadt \cite{Ham02} for further discussion 
of the Patterson-Sullivan geodesic current. 

Let $F(S)$ denote the space of Fuchsian representations in $QF(S)$ and let $\mathcal{PC}(S)$ denote
the space of projective classes of non-zero currents in $\mathcal C(S)$.
Bonahon \cite[Thm. 18]{Bon88} showed that if one considers the associated map into 
the space of projective geodesic currents $\bar\psi:F(S)\to \mathcal{PC}(S)$, then
the closure of $\bar\psi(F(S))$ is $\bar\psi(F(S))\cup PML(S)$ where $PML(S)$ is the space of projective classes of
measured laminations. (We recall that  the set $ML(S)$ of measured laminations in $\mathcal C(S)$ is exactly the
closure of the set of positive multiples of currents associated to simple curves, see Bonahon \cite[Sec. 4.3]{bonahon-ends}
for more details.)

We note that $L(\psi(\rho_1),\rho_2)$ can also be defined to be the length in $\rho_2$ of a random geodesic in $\rho_1$, i.e.
\begin{equation}
\label{Lformula}
L(\psi(\rho_1),\rho_2) =\lim_{T\rightarrow \infty} \frac{1}{\# \{[\gamma]\ | \ \ell_{\rho_1}(\gamma)\le T\} } 
\sum_{\{[\gamma]\ | \ \ell_{\rho_1}(\gamma)\le T\} }\frac{\ell_{\rho_2}(\gamma)}{\ell_{\rho_1}(\gamma)}.
\end{equation}
It follows  that
\begin{equation}
\label{Jformula}
\JJ(\rho_1,\rho_2)=\frac{h(\rho_2)}{h(\rho_1)}L(\psi(\rho_1),\rho_2).
\end{equation}

In analogy with Thurston's formulation of the Weil-Petersson metric on Teichm\"uller space, 
and Bonahon's re-formulation in terms of geodesic currents \cite[Thm. 19]{Bon88}, Bridgeman and Taylor \cite{BT08}
consider, for each $\rho\in QF(S)$, the function $\JJ_\rho:QF(S)\to \mathbb R$ given by $\JJ_\rho(\sigma)=\JJ(\rho,\sigma)$
and define a  non-negative symmetric two-tensor $G$, called the {\em pressure form}, on  the tangent bundle ${\sf T}QF(S)$,
where $G_\rho$  on ${\sf T}_\rho QF(S)$ is the Hessian of $\JJ_\rho$. By construction, the pressure form is invariant
under the action of the (extended) mapping class group ${\rm Mod}^*(S)$ of $S$.
It follows from the work of Wolpert \cite{Wol86}
and Bonahon \cite{Bon88}, that the restriction of the pressure form to the Fuchsian locus $F(S)$ is a multiple
of the Weil-Petersson metric.
Motivated by the work of McMullen \cite{mcmullen-pressure}
in the setting of Teichm\"uller space, Bridgeman \cite{Bri10} showed that the pressure form is non-degenerate
except along pure bending vectors based at points in the Fuchsian locus. Moreover, the pressure form
gives rise to a path metric on $QF(S)$, see \cite[Cor. 1.7]{BCLS}.

\medskip\noindent
{\bf Remark:} In the proof of Theorem 6.1 in \cite{Bri10}, Bridgeman gives an expression for $L(\psi(\rho_1),\rho_2)$ in
terms of equilibrium measures. Our equation (\ref{Lformula}) then follows from equations (12) and (13)  in \cite{BCLS}, 
see also the discussion
in section 8 of \cite{BCLS}. If one prefers that the proof of Theorem \ref{isom} be self-contained, one can take $J$ to be defined by
our equation (\ref{Jformula}) as is done in Bridgeman-Taylor \cite{BT08} and Bridgeman \cite{Bri10}.

\subsection{Isometries of renormalized pressure intersection}
We use the interpretation of renormalized pressure intersection in terms of geodesic currents to 
show that the restriction of an isometry $f$ of the renormalized pressure intersection
to the Fuchsian locus $F(S)$ agrees with an element $g$ of
the (extended) mapping class group. We further show that $f(\sigma)$ and $g(\sigma)$ have closely
related length spectrum whenever $\sigma\in QF(S)$.

\begin{prop}
If $f: QF(S) \rightarrow QF(S)$ is a smooth isometry of  $\JJ$, 
then $f$ preserves the Fuchsian locus $F(S)$ and there exists   $ g \in {\rm Mod}^*(S)$
such that $f$ and $g$ agree on the Fuchsian locus.

Moreover,  if $\alpha$ is a simple curve on $S$ and $\sigma\in QF(S)$, then
    $$h(g(\sigma)) \ell_{g(\sigma)}(\alpha)= h(f(\sigma)) \ell_{f(\sigma)}(\alpha).$$
\label{krel}
\end{prop}

\begin{proof}
Since $f$ is smooth and preserves $\JJ$,  one sees by differentiating that $f$ also preserves the pressure form $G$. 

If $\sigma\in F(X)$ and $v$ is a pure-bending vector, then $G_\sigma(v,v)=0$,
so \hbox{$G_{f(\sigma)}(Df_\sigma(v),Df_\sigma(v))=0$}. Since, $f$ is an immersion, $Df_\sigma(v)=0$
which implies that $f(\sigma)\in F(S)$.

Since the restriction of $G$ to $F(S)$ is a multiple of the Weil-Petersson metric, the restriction of $f$ to the
Fuchsian locus is an isometry of the Weil-Petersson metric. Masur and Wolf \cite{masur-wolf}
proved every isometry of the Weil-Petersson metric is an element of ${\rm Mod}^*(S)$.
So we may choose $g\in {\rm Mod}^*(S)$ so that $f$ and $g$ agree on the Fuchsian locus $F(S)$.

If $\rho \in F(S)$ and $ \sigma \in QF(S)$, then
since 
$$\JJ(\rho,\sigma) = \JJ(f(\rho), f(\sigma)) = \JJ(g(\rho),g(\sigma)),$$
and $h$ is constant on $F(S)$,
we may conclude that
$$\frac{L(\psi(g(\rho)),g(\sigma))}{L(\psi(f(\rho),f(\sigma))} =\frac{h(f(\sigma))}{h(g(\sigma))}.$$

If $\alpha$ is a simple curve, let  $\{\rho_n\}$ be a sequence in $F(S)$  so
that $\{\bar\psi(f(\rho_n))\}=\{\bar\psi(g(\rho_n))\}$ converges to $[C_\alpha]$ in $\mathcal{PC}(S)$.
Since $L$ is continuous,
$$\frac{\ell_{g(\sigma)}(\alpha)}{\ell_{f(\sigma)}(\alpha)}=\lim \frac{L(\psi(g(\rho_n)),g(\sigma))}{L(\psi(f(\rho_n)),f(\sigma))}
=\frac{h(f(\sigma))}{h(g(\sigma))}$$
which establishes our claim.
\end{proof}

\subsection{Simply $k$-related Kleinian surface groups}

We say  that two Kleinian surface groups $\rho_1:\pi_1(S)\to{\rm PSL}(2,\mathbb C)$ and 
\hbox{$\rho_2:\pi_1(S)\to {\rm PSL}(2,\mathbb C)$}  are {\em simply $k$-related}, 
for some $k>0$, if
$$\ell_{\rho_1}(\alpha) = k \ell_{\rho_2}(\alpha)$$ 
whenever $\alpha$ is a simple closed curve on $S$. 

Proposition \ref{krel} implies that if $f$ is an isometry of the renormalized pressure intersection,
then there exists $g\in {\rm Mod}^*(S)$, such that $f(\sigma)$ and $g(\sigma)$ are simply
$k$-related whenever $\sigma\in QF(S)$. Theorem \ref{isom} will thus follow from the following
strengthening of our simple length rigidity theorem.

\medskip\noindent
{\bf Theorem \ref{ksimple rigidity}:} {\em
If $S$ is a closed, connected, orientable surface of genus at least two and
\hbox{$\rho_1:\pi_1(S)\to {\rm PSL}(2,\mathbb C)$} and 
\hbox{$\rho_2:\pi_1(S)\to {\rm PSL}(2,\mathbb C)$} are  simply $k$-related Kleinian surface groups, 
then $\rho_1$ is conjugate to either $\rho_2$ or $\bar\rho_2$.}

\medskip\noindent
{\em Proof of Theorem \ref{ksimple rigidity}:}
We first choose non-separating simple  curves $\alpha$ and $\beta$ on $S$ which intersect
exactly once, so that $\rho_1(\alpha)$, $\rho_1(\beta)$, $\rho_2(\alpha)$ and $\rho_2(\beta)$ are all
hyperbolic. (One may  do so, since, again by a result of Sullivan \cite{sullivan-finite},
only finitely many simple curves have parabolic images for $\rho_1$ or $\rho_2$.)

We can assume that $\rho_1$ and $\rho_2$ are $(\alpha,\beta)$-normalized, so
$$\rho_i(\alpha)= \pm \left( \begin{array}{cc}
\lambda_i & 0\\
0 &   \lambda_i^{-1}\end{array}\right) 
$$
where $|\lambda_i|> 1$, and
$$\rho_i(\beta)= \pm \left( \begin{array}{cc}
a _i& b_i\\
c_i &  d_i \end{array}\right)$$
where all the matrix coefficients are non-zero (see Lemma \ref{basic comp}).
Since $\rho_1$ and $\rho_2$ are simply $k$-related,
$$|\lambda_1|=|\lambda_2|^k.$$

Notice that if $k=1$, then Theorem \ref{ksimple rigidity} follows immediately from
our main result, Theorem \ref{main}.
So we may assume, without loss of generality, that $k>1$.

Let $\mu_i(n)$  be the modulus of the  eigenvalue of $\rho_i(\alpha^n\beta)$ with largest modulus. 
Since $\alpha^n\beta$ is simple for all $n$, $\mu_1(n)= \mu_2(n)^k$ for all $n$.
Lemma \ref{basic comp} implies that 
$$\log\mu_i(n)= n\log|\lambda_i| + \log|a_i| +   \Re\left( \lambda_i^{-2n}\frac{b_ic_i}{a_i^{2}}\right) + O(|\lambda_i|^{-4n}).$$
Since $\log\mu_1(n) = k\log\mu_2(n)$, after eliminating the leading terms (which are equal), we obtain
$$ \log|a_1| + \Re\left( \lambda_1^{-2n}\frac{b_1c_1}{a_1^{2}}\right)) + O(|\lambda_1|^{-4n}) =
 k\log|a_2| + k\Re\left( \lambda_2^{-2n}\frac{b_2c_2}{a_2^{2}}\right) + O(|\lambda_2|^{-4n}).
$$
Therefore, by considering the limits as $n\to\infty$, we see that
$$\log|a_1| = k\log|a_2|.$$
So, after subtracting the equal  leading order terms,
$$  \Re\left( \lambda_1^{-2n}\frac{b_1c_1}{a_1^{2}}\right) + O(|\lambda_1|^{-4n}) =
 k\Re\left( \lambda_2^{-2n}\frac{b_2c_2}{a_2^{2}}\right) + O(|\lambda_2|^{-4n}).
$$
for all $n$.
If we multiplly both sides of the above equation by $|\lambda_2|^{2n}$, then
since $\frac{|\lambda_1|}{|\lambda_2|}>1$,  the right hand side converges to 0, and we see that 
\begin{equation}
\label{real part zero}
\lim_{n\rightarrow \infty}\Re\left( \left(\frac{\lambda_2}{|\lambda_2|}\right)^{-2n}\frac{b_2c_2}{a_2^{2}}\right)  = 0.
\end{equation}

If we let
$$u = \left(\frac{\lambda_{2}}{|\lambda_{2}|}\right)^{-2} = e^{i \theta} \quad\textrm{and}\quad v = \frac{b_2c_2}{a_2^{2}} \neq 0.$$
then we may rewrite (\ref{real part zero}) as 
$$\lim_{n\rightarrow \infty}\Re(u^{n}v) = 0.$$
As we can always choose a sequence $\{n_{k}\}$ so that $\lim_{k \rightarrow \infty} u^{n_{k}} = 1$, 
we see that
$$\Re(v) = 0.$$ 

If $\theta$ is  an irrational multiple of $2\pi$, then for all $z\in S^1$ there exists a sequence $\{n_k\}$ so that
$\lim e^{n_k\theta}=z$, which implies that $\Re(z v) = 0$.
Therefore, in this case, as $v \neq 0$ then choosing  $z = |v|/v$, we get $|v|=0$ a contradiction.

If $\theta$ is a rational multiple of $2\pi$ and $u\ne\pm 1$, then there exists $\{ n_k\}$ so that $u^{n_k}=u$ for all $k$,
so $\Re(uv) = 0$.
Since $\Re(v) = 0$, $v = iy$ from some $y \in \R$, so, 
$$\Re(e^{i\theta} v) = -\sin(\theta) y=0.$$
Thus, $y = 0$, since $\sin(\theta)\ne 0$. It again follows that $v = 0$ and
we obtain a contradiction.

Therefore, we can assume that  $u = \pm 1$, so
$\lambda^2_2$  is real. It follows that if $\eta$ is a simple, non-separating curve on $S$, then
$\lambda_{\rho_2}^2(\eta)\in\mathbb R$, so
$$\Tr^2(\rho_2(\eta)) = \lambda_{\rho_2}^2(\eta)+2+\lambda_{\rho_2}^{-2}(\eta)\in\mathbb R.$$ 

In particular, since $\alpha^n\beta$ is simple and non-separating for all $n$, we see that
$$\Tr^2(\rho_2(\alpha^n\beta)) 
= \lambda_2^{2n}a_2^{2} + 2a_2d_2 + \lambda_2^{-2n}d_2^2\in\mathbb R,$$
so, 
$$\Im(\Tr^2(\rho_2(\alpha^n\beta)) )
 = \lambda_2^{2n}\Im(a_2^{2}) + 2\Im(a_2d_2) + \lambda_2^{-2n}\Im(d_2^2)=0$$
 for all $n$.
Therefore,
$$\lim_{n \rightarrow \infty} \frac{\Im(\Tr^2(\rho_2(\alpha^n\beta))}{\lambda_{2}^{2n}} = \Im(a_2^{2})  = 0,$$
so
$$\lim_{n \rightarrow \infty} \Im(\Tr^2(\rho_2(\alpha^n\beta))) = 2\Im(a_2d_2) = 0.$$
Since $a_2^{2}$ and $a_2d_2$ are real,  $v =\frac{b_2c_2}{a_2^2}= \frac{a_2d_2-1}{a_2^2}$ is real. 
Therefore, since we have already shown that $\Re(v)=0$, we see that $v = 0$, which
contradicts the fact that $a_2$, $b_2$ and $c_2$ are all non-zero.
Since we have now achieved a contradiction in all cases where $k>1$, this completes the proof.
\qed

\subsection{Proof of Theorem \ref{isom}}
If $f$ is an isometry of the renormalized pressure intersection, then Theorem \ref{krel} implies
that there exists $g\in {\rm Mod}^*(S)$ such that $f$ and $g$ agree on $F(S)$ and $f(\sigma)$
and $g(\sigma)$ are simply $k$-related, where $k=\frac{h(f(\sigma))}{h(g(\sigma))}$, whenever
$\sigma\in QF(S)$.
Theorem \ref{ksimple rigidity} implies that if $\sigma \in QF(S)$, then either $f(\sigma)=g(\sigma)$ or
$f(\sigma)=\tau(g(\sigma))$.
Since $f$ is continuous, this implies that either $f = g$ or $f=\tau\circ g$.  The result follows.
\eproof

\end{document}